\documentclass[a4paper]{article}
\usepackage[utf8]{inputenc}
\usepackage{amsthm}
\usepackage{amsmath}
\usepackage{amssymb}
\usepackage{caption}
\usepackage{cite}
\usepackage{tikz}
\usepackage{subcaption}
\usepackage[labelformat=simple]{subcaption}
\usepackage{a4wide}
\usetikzlibrary{calc}

\usepackage{hyperref}
\newcommand{\footremember}[2]{%
    \footnote{#2}
    \newcounter{#1}
    \setcounter{#1}{\value{footnote}}%
}
\newcommand{\footrecall}[1]{%
    \footnotemark[\value{#1}]%
}

\usepackage{enumitem}
\setlist[enumerate]{noitemsep,topsep=3pt}
\setlist[itemize]{noitemsep,topsep=3pt}

\theoremstyle{definition}
\newtheorem{defin}{Definition}
\newtheorem{remark}[defin]{Remark}
\newtheorem{ex}[defin]{Example}
\theoremstyle{plain}
\newtheorem{lem}[defin]{Lemma}
\newtheorem{thm}[defin]{Theorem}
\newtheorem{cor}[defin]{Corollary}

\newcommand{\U}{\mathcal{U}}

\newcommand{\cg}[1]{\overline{#1}}

\author{%
  Anni Hakanen \footnote{Corresponding author, email: anehak@utu.fi. Research partially funded by the Magnus Ehrnrooth foundation.}  \footremember{TY}{Department of Mathematics and Statistics, University of Turku, Finland}%
  \and Ville Junnila \footrecall{TY}%
  \and Tero Laihonen \footrecall{TY}%
  \and Ismael G. Yero \footremember{UC}{Department of Mathematics, Universidad de C\'{a}diz, Spain} 
  }

\date{}

\title{On Vertices Contained in All or in No Metric Basis}

\begin{document}
\maketitle

\begin{abstract}
A set $R \subseteq V(G)$ is a \emph{resolving set} of a graph $G$ if for all distinct vertices $v,u \in V(G)$ there exists an element $r \in R$ such that $d(r,v) \neq d(r,u)$.
The \emph{metric dimension} $\dim(G)$ of the graph $G$ is the minimum cardinality of a resolving set of $G$.
A resolving set with cardinality $\dim(G)$ is called a \emph{metric basis} of $G$.
We consider vertices that are in all metric bases, and we call them basis forced vertices.
We give several structural properties of sparse and dense graphs where basis forced vertices are present.
In particular, we give bounds for the maximum number of edges in a graph containing basis forced vertices.
Our bound is optimal whenever the number of basis forced vertices is even.
Moreover, we provide a method of constructing fairly sparse graphs with basis forced vertices.
We also study vertices which are in no metric basis in connection to cut-vertices and pendants.
Furthermore, we show that deciding whether a vertex is in all metric bases is co-NP-hard, and deciding whether a vertex is in no metric basis is NP-hard. 
\end{abstract}

\section{Introduction}

Metric dimension in graphs is a classical invariant that has been studied from a lot of different points of view, starting from pure theoretical aspects and including some applied models mainly arising from computer science, among others.
Specifically, in the last decade, there has been an increasing number of investigations on the metric dimension of graphs, and a large number of them are specifically centered on variations of the standard metric dimension concept.
 The classical metric dimension parameter, although it is nowadays very well studied, strong interest still remains.
Indeed, a significant number of works specifically dealing with it appear frequently.
For some latest remarkable articles we suggest for instance \cite{GenesonPattern, GutinAlternative, JiangCartesian, LairdHamming, SedlarUnicyclic}.

For a given connected graph $G$, the distance $d_G(u,v)$ (the subindex can be removed if it is clear from context) between two vertices $u,v\in V(G)$ is the length of a shortest path between $u$ and $v$.
A set of vertices $R \subseteq V(G)$ is called a \emph{resolving set} for $G$ if for any pair of distinct vertices $u,v\in V(G)$, there exists a vertex $w\in R$ such that $d_G(w,u)\ne d_G(w,v)$.
In this sense, it is also said that the set $R$ \emph{resolves} the graph $G$, and that $w$ \emph{resolves} the vertices $u,v$ (or that $u,v$ are resolved by $w$).
The cardinality of the smallest possible resolving set for $G$ is
the \emph{metric dimension} of $G$, denoted by $\dim(G)$.
A resolving set of cardinality $\dim(G)$ is a \emph{metric basis} of $G$.
Such concepts were (independently) introduced in \cite{S:leavesTree} and \cite{Harary76}.

It is natural to ask if a graph contains a unique metric basis, or to ask a broader question, whether a graph has vertices belonging to every metric basis.
Such ideas were already investigated in \cite{BagheriUnique16, BuczkowskiUnique}.
For instance, it was shown in \cite{BuczkowskiUnique}, that for every integers $k,r$ with $k\ge 2$ and $0\le r\le k$, there exists a graph $G$ with metric dimension $k$ having $r$ vertices that belong to every metric basis of $G$.
However, no more contributions on this direction, concerning the classical metric dimension invariant, have appeared so far.
Such questions are widely studied in other contexts such as dominating sets and stable sets in \cite{Hammer82, CockayneHM03, Mynhardt99}, for example.
For some other variants of metric dimension, some related studies have continued.
Examples of this are \cite{Solid, HJLP:SetsOfVertices}, where forced vertices were defined as vertices that are in every $\ell$-solid-resolving set or $\{\ell\}$-resolving set -- not only the corresponding metric bases.
In this sense, it is now our goal to retrieve such ideas, for the classical metric dimension, and present an exposition of combinatorial results aimed to describe structural properties of graphs containing vertices that either belong to every, or to no metric basis of it.
For more formality in our exposition, we shall now name such kind of vertices in a graph $G$.

\begin{defin}\label{def:basisforced}
A vertex $v \in V(G)$ is a \emph{basis forced vertex} of the graph $G$ if it is contained in every metric basis of $G$.
\end{defin}

In connection with the existent models related to metric dimension in graphs, like that of navigation of robots in networks, the existence of basis forced vertices means that a metric basis cannot be formed without them.
In consequence, they need to be included at any possible set of ``landmarks'' used to uniquely identify the vertices of the graph.

Clearly, all graphs do not have basis forced vertices.
For example, the path $P_n$ has two disjoint metric bases (the endpoints), and thus no basis forced vertices.
The graph in Figure \ref{fig:exsensitiveunic} has two basis forced vertices, which are illustrated as black vertices.
This graph has a unique metric basis that consists of only these two vertices.
Consequently, the white vertices are in no metric basis.
These are the opposite of basis forced vertices, and thus we have the following definition.

\begin{defin}\label{def:void}
A vertex $v \in V(G)$ is a \emph{void vertex} of the graph $G$ if it is in no metric basis of $G$.
\end{defin}

In \cite{BorosCore}, the set of vertices belonging to all maximum stable sets of a graph is called the \emph{core} of $G$. In the same spirit, we could call the set of basis forced vertices a core for the metric bases of $G$ and the set of void vertices an anticore for the metric bases of $G$.

The vertices of any graph can be divided into three categories: basis forced vertices, void
vertices and vertices that do not fit in either category, that is, vertices that are in some metric
bases but not all.
The endpoints of a path are in some metric bases but not all, whereas the rest
of the vertices of a path are void vertices.
The vertices of a complete graph or a cycle are also all in some
bases but not all.
If a graph has a unique metric basis (like the graph in Figure \ref{fig:exsensitiveunic}), then each
of its vertices is either a basis forced vertex or a void vertex.
All three types of vertices can be
present simultaneously in the same graph.
This is the case for the graph later shown in Figure \ref{fig:sparseEx}.

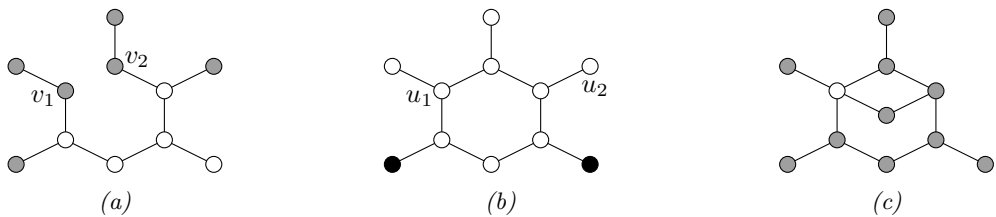
\begin{figure}[b]
\centering
\begin{subfigure}[b]{0.32\linewidth}
\centering
\begin{tikzpicture}[scale=.65]
 \draw (0,0) -- (1,.5) -- (1,1.5) -- (0,2) 
 		 (-1,1.5) -- (-1,.5) -- (0,0);
 \draw  (1,.5) -- (2,0)
 		(-1,.5) -- (-2,0)
 		(1,1.5) -- (2,2)
 		(-1,1.5) -- (-2,2)
 		(0,2) -- (0,3);
 \draw \foreach \x in {(0,0),(1,.5),(-1,.5),(1,1.5),(2,0)} {
 	\x node[circle, draw, fill=white,
             inner sep=0pt, minimum width=6pt] {}
 };
 \draw \foreach \x in {(0,2),(2,2),(-2,2),(0,3),(-2,0),(-1,1.5)} {
 	\x node[circle, draw, fill=gray!75,
             inner sep=0pt, minimum width=6pt] {}
 };
 \draw {
 	(-1.45,1.35) node[] {$v_1$}
 	(.45,2.15) node[] {$v_2$} };
\end{tikzpicture}
\caption{ }\label{fig:exsensitivetree}
\end{subfigure}
\hfill
\begin{subfigure}[b]{0.32\linewidth}
\centering
\begin{tikzpicture}[scale=.65]
 \draw (0,0) -- (1,.5) -- (1,1.5) -- (0,2) -- (-1,1.5) -- (-1,.5) -- (0,0);
 \draw  (1,.5) -- (2,0)
 		(-1,.5) -- (-2,0)
 		(1,1.5) -- (2,2)
 		(-1,1.5) -- (-2,2)
 		(0,2) -- (0,3);
 \draw \foreach \x in {(0,0),(1,.5),(-1,.5),(1,1.5),(-1,1.5),(0,2),(2,2),(-2,2),(0,3)} {
 	\x node[circle, draw, fill=white,
             inner sep=0pt, minimum width=6pt] {}
 };
 \draw \foreach \x in {(2,0),(-2,0)} {
 	\x node[circle, draw, fill=black,
             inner sep=0pt, minimum width=6pt] {}
 };
 \draw {
 	(-1.45,1.35) node[] {$u_1$}
 	(2.1,1.55) node[] {$u_2$} };
\end{tikzpicture}
\caption{ }\label{fig:exsensitiveunic}
\end{subfigure}
\hfill
\begin{subfigure}[b]{0.32\linewidth}
\centering
\begin{tikzpicture}[scale=.65]
 \draw (0,0) -- (1,.5) -- (1,1.5) -- (0,2) -- (-1,1.5) -- (-1,.5) -- (0,0);
 \draw  (1,.5) -- (2,0)
 		(-1,.5) -- (-2,0)
 		(1,1.5) -- (0,1) -- (-1,1.5)
 		(-1,1.5) -- (-2,2)
 		(0,2) -- (0,3);
 \draw (-1,1.5) node[circle, draw, fill=white,
             inner sep=0pt, minimum width=6pt] {};
 \draw \foreach \x in {(0,0),(1,.5),(-1,.5),(1,1.5),(2,0),(-2,0),(0,2),(0,1),(-2,2),(0,3)} {
 	\x node[circle, draw, fill=gray!75,
             inner sep=0pt, minimum width=6pt] {}
 };
\end{tikzpicture}
\caption{ }\label{fig:exsensitivecube}
\end{subfigure}
\caption{Void vertices and basis forced vertices are illustrated as white and black vertices vertices, respectively. Gray vertices are in some metric basis but not all.}
\end{figure}

In \cite{Solid} and \cite{HJLP:SetsOfVertices}, forced vertices of $\ell$-solid-resolving sets for all $\ell$, and $\{\ell\}$-resolving sets for $\ell\ge 2$ were characterised.
These characterisations used the local properties of the vertices and their neighbourhoods.
In contrast, it seems that achieving a similar local characterisation for basis forced vertices is a very challenging problem (see Section \ref{sec:Complexity} for issues concerning the algorithmic complexity).
For instance, if the graph changes even slightly, the metric bases may look very different even in places that are far away from where the change occurred.
This allows to realize that basis forced vertices are very sensitive to changes in the graph.
Consider the graph in Figure \ref{fig:exsensitivetree}.
The metric dimension of this graph is 2, and the vertices that appear in some metric basis are illustrated as gray vertices.
When we add the edge $v_1v_2$ we obtain the graph illustrated in Figure \ref{fig:exsensitiveunic}.
The metric dimension of this graph is also 2.
However, this graph has only one metric basis, the elements of which are illustrated as black vertices.
Consequently, this graph has two basis forced vertices and ten void vertices.
When we add the edge $u_1u_2$ to this graph, we obtain the graph in Figure \ref{fig:exsensitivecube}.
The metric dimension of this graph is 3, and there is only one void vertex.
The rest of the vertices appear in some metric bases but not all.

We next include some terminology and notations that we shall use throughout our exposition. 
Whenever two distinct vertices $v,u \in V(G)$ are adjacent, we denote $v \sim u$ (possibly with a subindex $\sim_G$ when such distinction is necessary), and when $v$ and $u$ are not adjacent, we denote $v \not\sim u$. 
The vertex $v$ is a \emph{pendant} if there exists exactly one element $u \in V(G) \setminus \{v\}$ such that $v \sim u$.
Let $S \subseteq V(G)$.
The \emph{induced subgraph} $G[S]$ is the graph with the vertex set $S$ and the edge set consisting of the edges of $G$ that are between two elements of $S$.
Let $v\in V(G)$.
The graph $G-v$ is obtained from $G$ by removing the vertex $v$ and all edges that have $v$ as an endpoint.
The vertex $v \in V(G)$ is a \emph{cut-vertex} of a connected graph $G$ if the graph $G-v$ is disconnected.
A connected graph $G$ is \emph{unicyclic} if it contains only one cycle (i.e., the number of edges is equal to the number of vertices in $G$).
From now on, all the graphs considered are connected.

The structure of this article is as follows. In Section \ref{sec:CutPends}, we consider cut-vertices and pendants. More precisely, in Section \ref{sec:Cut-vertices} we show that cut-vertices are not basis forced vertices, and that in fact most of them are void vertices. In Section \ref{sec:Pendants}, we show that pendants are not basis forced vertices in most cases. However, we also give an example of a graph that has pendants as basis forced vertices. In Section \ref{sec:SparseDense}, we consider sparse and dense graphs and their basis forced vertices. We investigate, what are the sparsest and densest graphs that have basis forced vertices. In Section \ref{sec:Colour}, we give two bounds on the number of basis forced vertices of a graph by using a new colouring method of visualising which pairs of vertices are resolved by which elements of a metric basis. Finally, we consider the algorithmic complexity of deciding whether a vertex is a basis forced or void vertex in Section \ref{sec:Complexity}.

\section{On Cut-Vertices and Pendants}\label{sec:CutPends}

In this section, we consider cut-vertices and pendants, and establish whether they can be void vertices or basis forced vertices. 
In the first subsection, we will show that in most cases cut-vertices are void vertices.
However, not all cut-vertices are void vertices.
For example, the set $\{v_1,v_2\}$ is a metric basis of the graph in Figure \ref{fig:exsensitivetree} and it consists of only cut-vertices.
We will show that the cut-vertices that are not void vertices are not basis forced either. 
In the latter subsection, we consider pendants.
We show that pendants are not basis forced vertices in most cases.
However, there are some graphs that have pendants as basis forced vertices.

\subsection{Cut-vertices}\label{sec:Cut-vertices}

Let us first prove a lemma that we will use in this section and the section on pendants.
This lemma essentially says that if there are elements of a resolving set in multiple connected components of $G-v$ for a cut-vertex $v$, then any pair of vertices that the cut-vertex $v$ resolves is also resolved by other elements of the resolving set.
However, the set $R$ is not required to be a resolving set for the following lemma to hold.

\begin{lem}\label{lem:potatocut}
Let $G$ be a connected graph with a cut-vertex $v$.
Let $R \subseteq V(G)$ be such that there are at least two connected components in $G - v$ with elements of $R$.
If $d(v,x) \neq d(v,y)$ for some $x,y \in V(G)$, then there exists an element $r \in R$ ($r \neq v$) such that $d(r,x) \neq d(r,y)$.
\end{lem}
\begin{proof}
Let us denote the connected components of $G - v$ by $G_i$, $i \in \mathbb{N}$, and let $R_i = R \cap V(G_i)$.
Assume that $x,y \in V(G)$ are such that $d(v,x) \neq d(v,y)$.
Suppose that for some $i$ we have $R_i \neq \emptyset$ and $x,y \notin V(G_i)$.
Now for all $r \in R_i$ we have
\[
d(r,x) = d(r,v) + d(v,x) \neq d(r,v) + d(v,y) = d(r,y).
\]
Thus, there exists an element $r \in R$ such that $d(r,x) \neq d(r,y)$.

Suppose then that no such $i$ exists.
Since there are at least two connected components with elements of $R$, we have $x \in V(G_i)$ and $y \in V(G_j)$, $i \neq j$, and $R_i, R_j \neq \emptyset$.
If for some $r \in R_i$ we have $d(r,x) \neq d(r,y)$, then we are done.
Suppose that for all $r \in R_i$ we have $d(r,x) = d(r,y)$.
Now
\[
d(r,v) + d(v,y) = d(r,y) = d(r,x) \leq d(r,v) + d(v,x).
\]
Since $d(v,y) \neq d(v,x)$, we have $d(v,y) < d(v,x)$.
However, now for all $s \in R_j$, we have 
\[
d(s,y) \leq d(s,v) + d(v,y) < d(s,v) + d(v,x) = d(s,x).
\]
Thus, $d(s,x) \neq d(s,y)$ for any $s \in R_j$.
\end{proof}

Showing that the cut-vertex $v$ is not in a metric basis is most difficult when there is a path dangling from $v$, i.e. a path with one end adjacent to $v$.
More precisely, there exists a connected component $G_i$ of $G-v$ such that $G[V(G_i) \cup \{v\}] \simeq P_n$ for some $n \in \mathbb{N}$.
This path-like structure does not necessarily contain any elements of a metric basis, so we need to be careful when using Lemma \ref{lem:potatocut}.
To help us differentiate the easy cases from the hard ones, we have the following lemma that already appeared in \cite{S:leavesTree}.
We include its proof by completeness of our work, and mainly based on the difficulty of finding reference \cite{S:leavesTree}.

\begin{lem}\label{lem:components} \cite[Proposition 2]{S:leavesTree}
Let $v$ be a cut-vertex of a connected graph $G$, and let $R$ be a resolving set of $G$.
For any connected component $G_i$ of $G - v$, we have 
\[
V(G_i) \cap R = \emptyset \quad \Rightarrow \quad G[V(G_i) \cup \{v\}] \simeq P_n
\]
for some $n \geq 2$.
Moreover, there can be only one connected component $G_i$ such that $V(G_i) \cap R = \emptyset$.
\end{lem}
\begin{proof}
Let $G_i$ be a connected component of $G - v$ such that $V(G_i) \cap R = \emptyset$.
Assume to the contrary that $G[V(G_i) \cup \{v\}] \not\simeq P_n$ for all $n \geq 2$.
Then there exist distinct $x,y \in V(G_i)$ such that $d(v,x) = d(x,y)$.
Consequently, $d(r,x) = d(r,v) + d(v,x) = d(r,v) + d(v,y) = d(r,y)$ for all $r \in R$.
Now the set $R$ is not a resolving set, a contradiction.

To prove the latter claim, suppose that $V(G_i) \cap R = \emptyset$ and $V(G_j) \cap R = \emptyset$ for some $i \neq j$.
Let $v_i \in N(v) \cap V(G_i)$ and $v_j \in N(v) \cap V(G_j)$.
Now, we have $d(r,v_i) = d(r,v) + 1 = d(r,v_j)$ for all $r \in R$, a contradiction.
\end{proof}

The following theorem considers the case where we can use Lemma \ref{lem:potatocut} straight away.
In Theorem \ref{thm:potatocutpath}, we consider the case not covered by Theorem \ref{thm:potatocutvoid}.

\begin{thm}\label{thm:potatocutvoid}
Let $G$ be a connected graph.
If $v \in V(G)$ is a cut-vertex such that 
\begin{enumerate}
\item $G - v$ has at least three connected components or 
\item $G - v$ has two connected components $G_i$ such that $G[V(G_i) \cup \{v\}] \not\simeq P_n$ for $i=1,2$,
\end{enumerate}
then $v$ is a void vertex. 
\end{thm}
\begin{proof}
Let $R$ be a resolving set of $G$ such that $v\in R$.
Let us denote the connected components of $G - v$ by $G_i$, $i \in \mathbb{N}$.
Due to our assumptions on $G - v$ and Lemma \ref{lem:components}, there are at least two connected components in $G - v$ that contain elements of $R$.

For any $x,y \in V(G)$ such that $d(v,x) = d(v,y)$ there exists an element $r \in R \setminus \{v\}$ such that $d(r,x) \neq d(r,y)$.
Suppose then that $x,y \in V(G)$ are such that $d(v,x) \neq d(v,y)$.
According to Lemma \ref{lem:potatocut}, for any such pair $x,y$ there exists an element $r \in R \setminus \{v\}$ such that $d(r,x) \neq d(r,y)$.
Thus, the set $R \setminus \{v\}$ is a resolving set, and $v$ is not in any metric basis.
\end{proof}

In many proofs, where we prove that some vertex is not a basis forced vertex, we use a replacing technique. 
Let $R$ be a metric basis of $G$.
Consider a fixed $r \in R$.
There exist some $x,y \in V(G)$ such that $d(r,x) \neq d(r,y)$ and $d(s,x) = d(s,y)$ for all $s \in R \setminus \{r\}$.
Otherwise, the set $R \setminus \{r\}$ would be a resolving set that is smaller than the metric basis $R$.
If there exists a vertex $v \in V(G)$ such that $d(v,x) \neq d(v,y)$ for all $x,y \in V(G)$ for which $r$ is the only element of $R$ that resolves them, then we can replace $r$ in $R$ with $v$ and obtain a metric basis of $G$ (i.e. the set $(R \setminus \{r\}) \cup \{v\}$ is a metric basis of $G$).
We denote $R[r \leftarrow v] = (R \setminus \{r\}) \cup \{v\}$.
Notice that if $v \in R$, then $R[r \leftarrow v] = R \setminus \{r\}$.

\begin{thm}\label{thm:potatocutpath}
Let $G$ be a connected graph.
Let $v \in V(G)$ be a cut-vertex such that $G - v $ has only two connected components $G_1$ and $G_2$.
If $G[V(G_i) \cup \{v\}] \simeq P_n$ for $i=1$ or $i=2$, then $v$ is not a basis forced vertex. 
\end{thm}
\begin{proof}
Suppose that $G[V(G_1) \cup \{v\}] \simeq P_n$ and that $R$ is a metric basis of $G$ such that $v \in R$.
Let $u$ be the vertex of $G_1$ that is furthest away from $v$ in $G$.

Let us show that $R[v \leftarrow u]$ is a metric basis of $G$.
To that end, let $x,y \in V(G)$ be such that $d(v,x) \neq d(v,y)$ (if $d(v,x) = d(v,y)$, then $d(r,x) \neq d(r,y)$ for some $r \in R \setminus \{v\}$).
It is sufficient to show that for every such $x$ and $y$ we also have $d(u,x) \neq d(u,y)$.
If $x \in V(G_1)$ or $y \in V(G_1)$ (or both), then we clearly have either $d(u,x) < d(u,y)$ or $d(u,y) < d(u,x)$, since the shortest paths from $u$ to $y$ go through $x$ or vice versa.
Suppose that $x,y \in V(G_2) \cup \{v\}$.
Now all the shortest paths from $u$ to $x$ and $y$ go through $v$.
Thus, we have
\[
d(u,x) = d(u,v) + d(v,x) \neq d(u,v) + d(v,y) = d(u,y).
\]
\end{proof}

According to Theorem \ref{thm:potatocutvoid} and the proof of Theorem \ref{thm:potatocutpath} a cut-vertex either is not in any metric basis or it can be replaced with a pendant.
Thus, the following corollary is immediate.

\begin{cor}\label{cor:potatocut}
Any finite graph has a metric basis that does not contain any cut-vertices.
\end{cor}

The graph has to be finite for Corollary \ref{cor:potatocut} to hold.
For example, the bi-infinite path consists of only cut-vertices, and thus any resolving set contains only cut-vertices.

\subsection{Pendants}\label{sec:Pendants}

Consider the graph in Figure \ref{fig:exsensitiveunic}.
This graph has a unique metric basis illustrated as black vertices in the figure.
The two pendants that form the unique metric basis are also basis forced vertices.
Thus, there does indeed exist graphs that have pendants as basis forced vertices.
However, our goal is to prove that in most cases the pendants are not basis forced vertices.

\begin{thm}\label{thm:potatopendgen}
Let $G$ be a connected graph with a pendant $u$.
Let $v$ be the cut-vertex adjacent to $u$.
If $G - \{v,u\}$ has at least two connected components, then $u$ is not a basis forced vertex.
\end{thm}
\begin{proof}
Let $R$ be a metric basis of $G$ such that $u \in R$.
Let us denote the connected components of $G - \{v,u\}$ by $G_i$, $i \in \mathbb{N}$, and let $R_i = R \cap V(G_i)$.
Since $R$ is a metric basis, there exists at least one pair of vertices $x,y \in V(G)$ such that $u$ is the only element in $R$ that resolves that pair.
In other words, $d(u,x) \neq d(u,y)$ and $d(r,x) = d(r,y)$ for all $r \in R \setminus \{u\}$.

\textbf{Case 1:} Suppose that there exist two connected components $G_i$ such that $R_i \neq \emptyset$.
Let us show that the pair $x,y$ is unique.
We will first prove that either $x=u$ or $y=u$.
Suppose to the contrary that $x \neq u$ and $y \neq u$.
Now 
\[
d(v,x) = d(u,x) - 1 \neq d(u,y) - 1 = d(v,y).
\]
According to Lemma \ref{lem:potatocut} (when applied to the set $R \setminus \{u\}$), the pair $x,y$ is resolved by some $r \in R \setminus \{u\}$, a contradiction.
Thus, either $x=u$ or $y=u$.

Suppose without loss of generality that $x=u$.
Let us show that $y$ is unique.
First of all, we have $y \neq v$, since any $r \in R \setminus \{u\}$ resolves $u$ and $v$.
Thus, $y \in V(G_i)$ for some $i$.
There exists some $j \neq i$ such that $R_j \neq \emptyset$.
Since $d(r,y) = d(r,u)$ for all $r \in R_j$, we have $y \in N(v) \cap V(G_i)$.
Thus, $d(u,y) = 2$.
Suppose that $y$ is not unique, that is, there exists $y' \in N(v) \setminus \{u,y\}$ such that $d(r,y') = d(r,u)$ for all $r \in R \setminus \{u\}$.
Now we cannot resolve $y$ and $y'$ with $R$: 
\begin{align*}
d(u,y) & = 2 = d(u,y') \text{ and} \\
d(r,y) & = d(r,u) = d(r,y')
\end{align*}
for all $r \in R \setminus \{u\}$.
Thus, $y$ is unique.

Since all pairs other than $u,y$ are resolved by $R \setminus \{u\}$, the set $R[u \leftarrow y]$ is a metric basis of $G$.
Therefore, $u$ is not a basis forced vertex.

\textbf{Case 2:} Suppose then that there is only one connected component $G_i$ such that $R_i \neq \emptyset$.
Due to Lemma \ref{lem:components}, there are only two connected components $G_1$ and $G_2$.
Suppose that $R_2 = \emptyset$.
Then $G[V(G_2) \cup \{v\}] \simeq P_n$ for some $n \in \mathbb{N}$.

Let $w \in V(G_2)$.
We will show that $R[u \leftarrow w]$ is a resolving set of $G$.
Let $x,y \in V(G)$ be such that $d(u,x) \neq d(u,y)$.
Suppose that $d(v,x) \neq d(v,y)$.
According to Lemma \ref{lem:potatocut} there exists an element $r \in R [u \leftarrow w]$ such that $d(r,x) \neq d(r,y)$.
Suppose then that $d(v,x) = d(v,y)$.
Since we assumed that $d(u,x) \neq d(u,y)$, we now have either $x = u$ or $y = u$. 

Suppose without loss of generality that $x = u$.
Then $y \in N(v)$.
If $y \in V(G_2)$, then $d(w,y) < d(w,u)$, since the shortest path from $w$ to $u$ goes through $y$.
Suppose then that $y \in V(G_1)$.
Denote by $y'$ the unique element in $N(v) \cap V(G_2)$.
Since $R$ is a metric basis and $d(u,y) = d(u,y')$, we have $d(r,y) \neq d(r,y')$ for some $r \in R \setminus \{u\}$.
Since $R_2 = \emptyset$, we have $r \in R_1$.
Thus, $d(r,u) = d(r,y') \neq d(r,y)$.
Therefore, if $R_2 = \emptyset$, then the set $R[u \leftarrow w]$ is a metric basis for any $w \in V(G_2)$.
\end{proof}

In the proofs of Theorems \ref{thm:potatocutvoid} and \ref{thm:potatocutpath}, our goal was to replace the cut-vertex to something else or remove it entirely from the resolving set.
However, in the following theorem we do the opposite; we replace a pendant with a cut-vertex in a metric basis.

\begin{thm}\label{thm:potatopendant2}
Let $G$ be a connected graph such that $G \not\simeq P_n$.
Let $u$ be a pendant adjacent to a cut-vertex $v$ such that $\deg (v) = 2$.
Then $u$ is not a basis forced vertex.
\end{thm}
\begin{proof}
Suppose that $R$ is a metric basis of $G$ such that $u \in R$.
We will show that $R[u \leftarrow v]$ is a resolving set.

Let $x,y \in V(G)$ be such that $d(u,x) \neq d(u,y)$.
If $x,y \in V(G) \setminus \{u\}$, then $d(v,x) = d(u,x) - 1 \neq d(u,y) - 1 = d(v,y)$.
Suppose that $x=u$ and $d(v,x) = d(v,y)$.
Since $G \not\simeq P_n$, there exists an element $r \in R \setminus \{v,u\}$ due to Lemma \ref{lem:components}.
Since $\deg (v) = 2$ and $y \in N(v)$, the shortest paths from $r$ to $u$ go through $y$.
Thus, $d(r,u) \neq d(r,y)$.
\end{proof}

Based on the results on the metric bases of trees presented in \cite{Chartrand00, Khuller96, S:leavesTree}, we can see that trees do not have any basis forced vertices.
Now we can also prove that trees have no basis forced vertices with the results obtained in Sections \ref{sec:Cut-vertices} and \ref{sec:Pendants}.
According to Corollary \ref{cor:potatocut} cut-vertices are not basis forced vertices.
If the tree in question is not $P_2$, then any pendant is adjacent to a cut-vertex that fulfils the requirements of either Theorem \ref{thm:potatopendgen} or Theorem \ref{thm:potatopendant2}.
Thus, neither the pendants of a tree are basis forced vertices.

Let us then consider a general graph $G$.
Suppose that $G$ has a cut-vertex $v$.
If there is a path- or tree-like structure attached to $v$, then there are no basis forced vertices in this structure.
More precisely, let $G_i$ be a connected component of $G-v$ such that $G_i$ is isomorphic to a tree or a path of length at least 2.
According to Corollary \ref{cor:potatocut} and Theorems \ref{thm:potatopendgen} and \ref{thm:potatopendant2}, $G_i$ does not contain basis forced vertices.

Suppose that there is a pendant $u$ adjacent to $v$.
If there is a path- or tree-like structure attached to $v$, then the graph fulfils the conditions of Theorem \ref{thm:potatopendgen} or \ref{thm:potatopendant2} and $u$ is not a basis forced vertex.
Thus, if $u$ is a basis forced vertex, then every edge $\{v,x\} \in E(G)$ where $x \in V(G) \setminus \{u\}$ is along a cycle.
However, all pendants of the graph in Figure \ref{fig:exsensitiveunic} satisfy this condition, but only two of them are basis forced vertices and the rest are void vertices.

\section{Sparse and Dense Graphs}\label{sec:SparseDense}

In this section, we again consider connected graphs.
We describe a graph as sparse if it has few edges compared to the number of vertices.
Conversely, a graph is dense if it has many edges compared to the number of vertices.
We want to find out what restrictions the existence of basis forced vertices places on the number of edges.

\subsection{Sparse Graphs}\label{sec:Sparse}

As we saw in Section \ref{sec:CutPends}, trees do not have any basis forced vertices.
However, the graph in Figure \ref{fig:exsensitiveunic} has two basis forced vertices.
Thus, the sparsest graphs that have basis forced vertices are unicyclic graphs.

The basis forced vertices of the graph in Figure \ref{fig:exsensitiveunic} are pendants.
A unicyclic graph can also have a basis forced vertex along the cycle; the vertex $u$ of the graph in Figure \ref{fig:sparseEx} is such a vertex.
All metric bases of this graph are of the form $\{u,v_i,v_j\}$, where $i,j \in \{1,2,3\}$ and $i \neq j$.
Both of these example graphs have only one or two basis forced vertices.
Indeed, later in Section \ref{sec:unicyclic2} we obtain a result which states that a unicyclic graph can have at most two basis forced vertices.

We want to construct connected graphs with, say, $k$ basis forced vertices and as few edges as possible.
In this endeavour, the following construction will be useful.
The general idea and the structure of the end-result of this construction is illustrated in Figure \ref{fig:constIdea}.
Recall that $\dim (G) = 1$ if and only if $G=P_n$ \cite{Chartrand00,Khuller96}.
Thus, the resolving sets of the components $G_i$ in the following theorem have at least two elements.

\begin{thm}\label{thm:construction}
Let $G_i$ be a connected graph such that $G_i \not\simeq P_n$ for $i = 1, \ldots , k$, and let all $G_i$ be vertex disjoint.
Let $g_i$ be a fixed element of $V(G_i)$ such that $g_i$ is in some metric basis of $G_i$ (i.e. $g_i$ is not a void vertex).
Let $W$ be the graph with 
\[
V(W) = \bigcup\limits^{k}_{i=1} V(G_i)
\quad \text{and} \quad
E(W) = \{\{g_i, g_j\} \ | \ i \neq j\} \cup \bigcup\limits^{k}_{i=1} E(G_i),
\]
i.e. the graph we obtain by connecting every $g_i$ with one another. 
A set $R \subseteq V(W)$ is a metric basis of $W$ if and only if 
\[
R = \bigcup\limits_{i=1}^{k} R_i \setminus \{g_i\}
\]
 where $R_i$ is a metric basis of $G_i$ that contains $g_i$ for each $i = 1, \ldots , k$.
Consequently, 
\[ \dim (W) = \sum\limits_{i=1}^k \dim (G_i) -k. \]
\end{thm}
\begin{figure}
\centering
\begin{subfigure}[b]{0.45\linewidth}
\centering
\begin{tikzpicture}[scale=.65]
 \draw (0,0) -- (1,.5) -- (1,1.5) -- (0,2) -- (-1,1.5) -- (-1,.5) -- (0,0);
 \draw (1,1.5) -- (2,2) -- (3,2)
 	   (2.7,1.3) -- (2,2) -- (2.7,2.7)
 	   (0,2) -- (0,3)
 	   (-1,1.5) -- (-2,2);
 \draw \foreach \x in {(0,0),(1,.5),(-1,.5),(1,1.5),(-1,1.5),(0,2),(2,2),(-2,2),(0,3)} {
 	\x node[circle, draw, fill=white,
             inner sep=0pt, minimum width=6pt] {}
 };
 \draw (0,0) node[circle, draw, fill=black,
             inner sep=0pt, minimum width=6pt] {};
 \draw \foreach \x in {(3,2),(2.7,1.3),(2.7,2.7)} {
 	\x node[circle, draw, fill=gray!75,
             inner sep=0pt, minimum width=6pt] {}
 };
 \draw {
	(0,-.4) node[] {$u$}
	(3.5,1.9) node[] {$v_2$}
	(3.2,1.2) node[] {$v_3$}
	(3.2,2.6) node[] {$v_1$}
 };
\end{tikzpicture}
\caption{A unicyclic graph with a basis forced vertex and metric dimension 3.}\label{fig:sparseEx}
\end{subfigure}
\hfill
\begin{subfigure}[b]{0.45\linewidth}
\centering
\begin{tikzpicture}[scale=.7]
 \draw (-.7,.3) -- (.7,.3) -- (-.7,1.7) -- (.7,1.7) -- (-.7,.3) 
 		-- (-.7,1.7)  (.7,.3) -- (.7,1.7);
 \draw[dashed]{
  (-1,0) ellipse (.8cm and .8cm)
  (1,0) ellipse (.8cm and .8cm)
  (-1,2) ellipse (.8cm and .8cm)
  (1,2) ellipse (.8cm and .8cm)
 };
 \draw \foreach \x in {(-.7,.3),(.7,.3),(-.7,1.7),(.7,1.7)} {
 	\x node[circle, draw, fill=white,
             inner sep=0pt, minimum width=5pt] {}
 };
 \draw {
  (-2.3,0) node[] {$G_3$}
  (2.3,0) node[] {$G_4$}
  (-2.3,2) node[] {$G_1$}
  (2.3,2) node[] {$G_2$}
 };
 \draw {
  (-1,0) node[] {$g_3$}
  (1,0) node[] {$g_4$}
  (-1,2) node[] {$g_1$}
  (1,2) node[] {$g_2$}
 };
\end{tikzpicture}
\caption{An example of the idea behind the construction in Theorem \ref{thm:construction}.}\label{fig:constIdea}
\end{subfigure}
\caption{ }
\end{figure}
\begin{proof}
We will first show that a set $R \subseteq V(W)$ is a resolving set of $W$ if and only if 
\[ R = \bigcup\limits_{i=1}^{k} R_i \setminus \{g_i\} \]
where $R_i$ is a resolving set of $G_i$ that contains $g_i$ for each $i = 1, \ldots , k$. 

Let $R \subseteq V(W)$ be a resolving set of $W$.
Since the vertices $g_i$ are cut-vertices of $W$ and $G_i \not\simeq P_n$ for all $i$, we have $S_i = R \cap V(G_i) \neq \emptyset$ for all $i$ due to Lemma \ref{lem:components}. 

Let us show that the set $S_i' = S_i \cup \{g_i\}$ is a resolving set of $G_i$.
Suppose to the contrary that $S_i'$ is not a resolving set of $G_i$.
Then there exist $x,y \in V(G_i)$ such that $d(s',x)=d(s',y)$ for all $s' \in S_i'$.
However, now the set $R$ is not a resolving set of $W$; we have $d(s_i,x) = d(s_i,y)$ for all $s_i \in S_i$ and 
\[ 
d(s_j,x) = d(s_j,g_i) + d(g_i,x) = d(s_j,g_i) + d(g_i,y) = d(s_j,y) 
\]
for all $s_j \in S_j$.

Let $R_i$ be a resolving set of $G_i$ that contains $g_i$ for all $i =1, \ldots , k$.
Let 
\[
R = \bigcup\limits_{i=1}^{k} R_i \setminus \{g_i\}.
\]
Let us show that $R$ is a resolving set of $W$ by showing that we can resolve every pair $x,y \in V(W)$.
We divide the proof into two cases:
\begin{enumerate}
\item $x,y \in V(G_i)$ for some $i \in \{1, \ldots , k\}$:
Since $R_i$ is a resolving set of $G_i$, there exists some $r_i \in R_i$ such that $d(r_i,x) \neq d(r_i,y)$.
If $r_i \neq g_i$, then $r_i \in R$ and we are done.
Suppose that $r_i = g_i$.
Now we can resolve the pair $x,y$ with any $r_j \in R_j$ where $j \neq i$.
Indeed, we have $d(r_j,x) = d(r_j,g_i) + d(g_i,x) \neq d(r_j,g_i) + d(g_i,y) = d(r_j,y)$ for all $r_j \in R_j$.

\item $x \in V(G_i)$ and $y \in V(G_j)$ for some $i \neq j$:
In this case, we use Lemma \ref{lem:potatocut}.
Each vertex $g_i$ is a cut-vertex in $W$.
Suppose that $d(g_i,x) = d(g_i,y)$.
Now, $d(g_j,x) = d(g_i,x) + 1 \neq d(g_i,y) - 1 = d(g_j,y)$.
Similarly, if $d(g_j,x) = d(g_j,y)$, then $d(g_i,x) \neq d(g_i,y)$. 
The graph $W-g_i$ has two connected components, neither of which is isomorphic to a path, since $G_i \not\simeq P_n$ for all $i$.
Thus, both connected components of $W - g_i$ (and similarly $W - g_j$) contain elements of $R$ according to Lemma \ref{lem:components}.
Thus, we can use Lemma \ref{lem:potatocut} for either $g_i$ or $g_j$ and obtain that for some $r \in R$ we have $d(r,x) \neq d(r,y)$. 
\end{enumerate}

Let us then consider the cardinalities of these resolving sets.
If $R$ is a metric basis of $W$, then by the equivalence shown above, we have 
$ \dim (W) = |R| \geq \sum\limits_{i=1}^k \dim (G_i) -k. $
Conversely, we can choose the resolving sets $R_i$ to be metric bases of $G_i$ for all $i=1, \ldots ,k$, and thus 
$ \dim (W) \leq \sum\limits_{i=1}^k \dim (G_i) -k. $
In conclusion, we have $ \dim (W) = \sum\limits_{i=1}^k \dim (G_i) -k$ and a set $R$ is a metric basis of $W$ if and only if 
$ R = \bigcup\limits_{i=1}^{k} R_i \setminus \{g_i\},$
where $R_i$ is a metric basis of $G_i$ that contains $g_i$ for each $i = 1, \ldots , k$.
\end{proof}

\begin{ex}
Consider the graph $G$ in Figure \ref{fig:ex:const}.
This graph has only three metric bases: $\{v_2,v_5\}$, $\{v_3,v_6\}$ and $\{v_5,v_6\}$.
Let us construct a new graph $W$ by using Theorem \ref{thm:construction} and two copies of $G$.
Let us choose $g_1=v_3$ and $g_2=v_2$.
Since there is only one metric basis that contains $v_3$ and one metric basis that contains $v_2$, the only metric basis of $W$ is $\{v^1_5,v^2_6\}$, where the superscripts indicate in which copy of $G$ the vertex is.
Thus, $W$ is a graph with a unique metric basis and metric dimension 2.
Consequently, the graph $W$ contains two basis forced vertices.

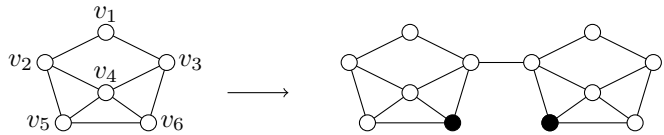
\begin{figure}[b]
\centering
\begin{tikzpicture}[scale=.8]
\draw (0,0) -- (-.7,-.5) -- (-1,.5) -- (0,1) -- (1,.5) -- (.7,-.5) 
		-- (0,0) -- (-1,.5)  (1,.5) -- (0,0)  (-.7,-.5) -- (.7,-.5);
\draw \foreach \x in {(0,0),(-.7,-.5),(.7,-.5),(1,.5),(-1,.5),(0,1)}{
    	\x node[circle, draw, fill=white,
       inner sep=0pt, minimum width=6pt] {} };
\draw {
	(0,.3) node[] {$v_4$}
	(-1.1,-.5) node[] {$v_5$}
	(1.1,-.5) node[] {$v_6$}
	(1.4,.5) node[] {$v_3$}
	(-1.4,.5) node[] {$v_2$}
	(0,1.3) node[] {$v_1$} };
\draw[->] (2,0) -- (3,0);
\begin{scope}[shift={(5,0)}]
\draw (0,0) -- (-.7,-.5) -- (-1,.5) -- (0,1) -- (1,.5) -- (.7,-.5) 
		-- (0,0) -- (-1,.5)  
		(2,.5) -- (1,.5) -- (0,0)  (-.7,-.5) -- (.7,-.5);
\draw \foreach \x in {(0,0),(-.7,-.5),(1,.5),(-1,.5),(0,1)}{
    	\x node[circle, draw, fill=white,
       inner sep=0pt, minimum width=6pt] {} };
\draw (.7,-.5) node[circle, draw, fill=black,
       inner sep=0pt, minimum width=6pt] {};
\end{scope}
\begin{scope}[shift={(8,0)}]
\draw (0,0) -- (-.7,-.5) -- (-1,.5) -- (0,1) -- (1,.5) -- (.7,-.5) 
		-- (0,0) -- (-1,.5)  (1,.5) -- (0,0)  (-.7,-.5) -- (.7,-.5);
\draw \foreach \x in {(0,0),(-.7,-.5),(.7,-.5),(1,.5),(-1,.5),(0,1)}{
    	\x node[circle, draw, fill=white,
       inner sep=0pt, minimum width=6pt] {} };
\draw (-.7,-.5) node[circle, draw, fill=black,
       inner sep=0pt, minimum width=6pt] {};
\end{scope}
\end{tikzpicture}
\caption{An example where we construct a graph with 2 basis forced vertices from a graph that has no basis forced vertices.}\label{fig:ex:const}
\end{figure}
\end{ex}

In general, there are two ways to use Theorem \ref{thm:construction}.
If we want to use $k$ graphs, we can connect them all in one go or iterate the use of Theorem \ref{thm:construction}.
In the first option, the vertices $g_i$ induce a clique in the constructed graph.
In the latter option, we first use Theorem \ref{thm:construction} on two graphs, then choose a new $g_1$ from the resulting graph and connect another graph to that.
We can also use a combination of these two methods.
The following example further demonstrates the difference of these two options.

\begin{ex}
Consider the graph $G$ in Figure \ref{fig:sparseEx}.
Let us combine three of these graphs by using Theorem \ref{thm:construction}.
Let us indicate with a superscript from which copy of the graph $G$ the vertices are from.
We choose $g_1 = v_1^1$, $g_2 = v_1^2$ and $g_3 = v_1^3$.
The constructed graph $W_1$ is illustrated in Figure \ref{fig:constclique}.
There are two metric bases of $G$ that contain $v_1$: $\{u,v_1,v_2\}$ and $\{u,v_1,v_3\}$.
Thus, the metric bases of $W_1$ are of the form $\{u^1,v_j^1,u^2,v_k^2,u^3,v_l^3\}$ where $j,k,l \in \{2,3\}$.
Thus, the metric dimension of $W_1$ is 6, and it has three basis forced vertices.

Let us then iterate Theorem \ref{thm:construction}.
Let us first use Theorem \ref{thm:construction} on two copies of $G$.
We again choose $g_1=v_1^1$ and $g_2=v_1^2$.
The metric bases of the resulting graph $W_2$ are of the form $\{u^1,v_j^1,u^2,v_k^2\}$ where $j,k \in \{2,3\}$.
Let us then use Theorem \ref{thm:construction} again, but this time on the graph $W_2$ and another copy of $G$.
We choose $g_1=v_3^2$ and $g_2=v_1^3$.
The resulting graph $W_3$ is illustrated in Figure \ref{fig:constpath}.
The only metric bases of $W_2$ that contain $v_3^2$ are of the form $\{u^1,v_j^1,u^2,v_3^2\}$ where $j \in \{2,3\}$.
Thus, the metric bases of $W_3$ are of the form $\{u^1,v_j^1,u^2,u^3,v_l^3\}$ where $j,l \in \{2,3\}$.
Now, we have constructed a graph with metric dimension 5 and three basis forced vertices.

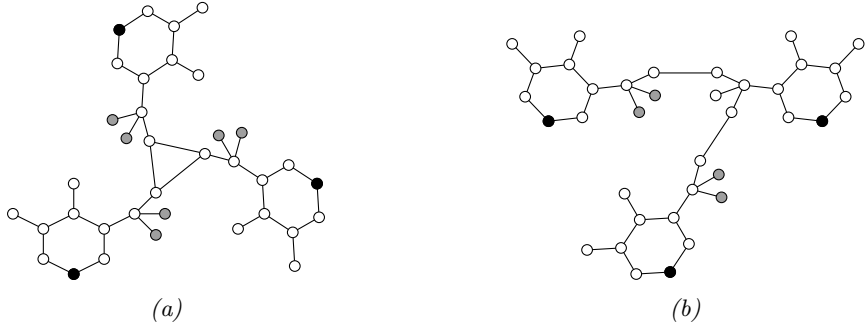
\begin{figure}
\centering
\begin{subfigure}[b]{0.45\linewidth}
\centering
\begin{tikzpicture}[scale=.8]
\def\graafi{ 
 \draw (0,0) -- (1,.5) -- (1,1.5) -- (0,2) -- (-1,1.5) -- (-1,.5) -- (0,0);
 \draw (1,1.5) -- (2,2) -- (3,2)
 	   (2.7,1.3) -- (2,2) -- (2.7,2.7)
 	   (0,2) -- (0,3)
 	   (-1,1.5) -- (-2,2);
 \draw \foreach \x in {(0,0),(1,.5),(-1,.5),(1,1.5),(-1,1.5),(0,2),(2,2),(-2,2),(0,3),(2.7,2.7)} {
 	\x node[circle, draw, fill=white,
             inner sep=0pt, minimum width=4pt] {}
 };
 \draw (0,0) node[circle, draw, fill=black,
             inner sep=0pt, minimum width=4pt] {};
 \draw \foreach \x in {(3,2),(2.7,1.3)} {
 	\x node[circle, draw, fill=gray!75,
             inner sep=0pt, minimum width=4pt] {} };
 };
 \draw (1.35,1.35) -- (2.15,2) -- (1.25,2.2) -- cycle;
 \begin{scope}[scale=.5]
 	\graafi \end{scope}
 \begin{scope}[scale=.5,shift={(8,3)},rotate=120]
 	\graafi \end{scope}
 \begin{scope}[scale=.5,shift={(1.5,8.1)},rotate=240]
 	\graafi \end{scope}
\end{tikzpicture}
\caption{ }\label{fig:constclique}
\end{subfigure}
\begin{subfigure}[b]{0.45\linewidth}
\centering
\begin{tikzpicture}[scale=.8]
\def\graafi{ 
 \draw (0,0) -- (1,.5) -- (1,1.5) -- (0,2) -- (-1,1.5) -- (-1,.5) -- (0,0);
 \draw (1,1.5) -- (2,2) -- (3,2)
 	   (2.7,1.3) -- (2,2) -- (2.7,2.7)
 	   (0,2) -- (0,3)
 	   (-1,1.5) -- (-2,2);
 \draw \foreach \x in {(0,0),(1,.5),(-1,.5),(1,1.5),(-1,1.5),(0,2),(2,2),(-2,2),(0,3),(2.7,2.7)} {
 	\x node[circle, draw, fill=white,
             inner sep=0pt, minimum width=4pt] {}
 };
 \draw (0,0) node[circle, draw, fill=black,
             inner sep=0pt, minimum width=4pt] {};
 \draw \foreach \x in {(3,2),(2.7,1.3)} {
 	\x node[circle, draw, fill=gray!75,
             inner sep=0pt, minimum width=4pt] {} };
 };
 \draw (1.7,.8) -- (2.8,.8)  (3.025,.15) -- (2.5,-.65);
 \begin{scope}[scale=.5,rotate=-20]
 	\graafi \end{scope}
 \begin{scope}[scale=.5,shift={(9,0)},xscale=-1,rotate=-20]
 	\draw (0,0) -- (1,.5) -- (1,1.5) -- (0,2) -- (-1,1.5) 
 		-- (-1,.5) -- (0,0);
 	\draw (1,1.5) -- (2,2) -- (3,2)
 	   (2.7,1.3) -- (2,2) -- (2.7,2.7)
 	   (0,2) -- (0,3)
 	   (-1,1.5) -- (-2,2);
 	\draw \foreach \x in {(0,0),(1,.5),(-1,.5),(1,1.5),(-1,1.5),(0,2),
 		(2,2),(-2,2),(0,3),(2.7,2.7),(3,2),(2.7,1.3)} {
 	\x node[circle, draw, fill=white,
             inner sep=0pt, minimum width=4pt] {}
 	};
 	\draw (0,0) node[circle, draw, fill=black,
             inner sep=0pt, minimum width=4pt] {}; 	
 \end{scope}
 \begin{scope}[scale=.5,shift={(4,-5)},rotate=30]
 	\graafi \end{scope}
\end{tikzpicture}
\caption{ }\label{fig:constpath}
\end{subfigure}
\caption{An example of the two ways we can use Theorem \ref{thm:construction}.}
\end{figure}

The graph $W_3$ has one edge less than the graph $W_1$.
If we use the construction on more graphs, then this difference will only grow.
Indeed, if we used $m$ copies of the graph $G$, then the first method would produce $m(m-1)$ edges to the resulting graph in addition to the edges already present in $G$.
Thus, the resulting graph would have $12m + m(m-1)$ edges.
However, if we iterate the construction, then each step adds only one edge in addition to those present in $G$.
If we iterate the construction on $m$ copies of $G$, we will use Theorem \ref{thm:construction} $m-1$ times, and the resulting graph has only $12m + m - 1$ edges. 
\end{ex}

We have now demonstrated a way to construct a graph that has $k$ basis forced vertices, $n=12k$ vertices and $n + k-1$ edges.
So far, we have not found any sparser graph with $k$ basis forced vertices.

\subsubsection{An Upper Bound for the Number of Basis Forced Vertices of Unicyclic Graphs}\label{sec:unicyclic2}

According to the terminology given in \cite{PoissonUnicyclic}, a unicyclic graph of type 1 is a unicyclic graph $G$ of maximum degree 3 and such that every vertex of maximum degree belongs to the unique cycle of $G$.
Any other unicyclic graph is called there as of type 2.
It was proved in \cite{PoissonUnicyclic} that any unicyclic graph of type 1 with unique cycle $C_n$ has metric dimension 2 when $n$ is odd, and otherwise its metric dimension is between 2 and 3.
We notice the following straightforward observation (see \cite{PoissonUnicyclic}).

\begin{remark}
If $G$ is a unicyclic graph of type 1 with unique cycle $C=v_0v_1\cdots v_{r-1}v_0$, $r\ge 3$, then any set $S=\{v_i,v_{i+1},v_{i+\left\lfloor r/2\right\rfloor}\}$ (operations with subindex are done modulo $r$) is a resolving set for $G$.
\end{remark}

As a consequence of the remark above, if $G$ is a unicyclic graph of type 1 with metric dimension 3, then $G$ does not contain basis forced vertices.
Thus, a unicyclic graph of type 1 can have at most two basis forced vertices.
We next consider unicyclic graphs of type 2.
From \cite{PoissonUnicyclic}, lower and upper bounds for the metric dimension of such unicyclic graphs are known.
However, from \cite{SedlarUnicyclic}, we can easier proceed with our deduction.
For a unicyclic graph $G$, let $b(G)$ be the number of vertices on the cycle that have something other than one pendant or path attached to it outside of the cycle.
Furthermore, let 
\[ L(G) = \sum\limits_{v \in V(G), \ell (v) >1} (\ell (v) - 1), \]
where $\ell (v)$ is the number of pendants and paths attached to $v$.
The following result was proved in \cite{SedlarUnicyclic}.

\begin{thm}{\em\cite{SedlarUnicyclic}}\label{unicyclic}
Let $G$ be an unicyclic graph.
Then $\dim(G)$ has value $L(G)+\max\{2-b(G),0\}$ or $L(G)+\max\{2-b(G),0\}+ 1$.
\end{thm}

We must remark that the set of vertices of $G$ that contributes to the value of $L(G)$ is a set that can be partitioned into subsets each of cardinality at least two, and one can always leave any one vertex from each of these subsets out of a metric basis of $G$. 
These facts and the way of computing $L(G)$, allow to deduce that no vertex of such set is basis forced. 
Thus, since $b(G)\ge 0$, we have $0\le \max\{2-b(G),0\}\le 2$, and the two possible values for $\dim(G)$, from the theorem above, leads to observe that $G$ could have between 0 and 3 basis forced vertices. 
However, if there are 3 basis forced vertices in $G$, then it must happen that $\dim(G)=L(G)+\max\{2-b(G),0\}+ 1$ with $b(G)=0$. 
But then it follows that every vertex of $G$ has degree at most 3 and every vertex of degree 3 belongs to the unique cycle of $G$. 
Consequently, $G$ is a unicyclic graph of type 1 (according to \cite{PoissonUnicyclic}), and we have seen that such graphs can have at most 2 basis forced vertices.
Thus, we have the following corollary.

\begin{cor}
If $G$ is a unicyclic graph, then it can have at most 2 basis forced vertices.
\end{cor}

Unicyclic graphs with 1 and 2 basis forced vertices are illustrated in Figures \ref{fig:sparseEx} and \ref{fig:exsensitiveunic}, respectively.

\subsection{Dense Graphs}\label{sec:Dense}

We want to find out, how many edges a graph can have and still contain basis forced vertices.
We will show that a graph with $n$ vertices and some basis forced vertices can have at most $\frac{n(n-1)}{2} -4$ edges.
We will also show that the graph attaining this bound is unique for each $n$.
To obtain these results, we will consider graphs that can be constructed from $K_n$ by removing up to four edges.
The relevant edge deletion patterns (i.e. connected components of the complement graph) are illustrated in Figure \ref{fig:denseclean}.

Distinct vertices $v,u \in V(G)$ are \emph{true twins} if $N[v]=N[u]$, and \emph{false twins} if $N(v)=N(u)$.
We simply say that two vertices are twins if they are true or false twins but it does not matter which. 
As the graphs we consider are quite dense, the graphs will most likely contain twins.
The following lemma follows directly from Corollary 2.4 of \cite{Hernando10}, although an independent proof would not be too difficult either.

\begin{lem}\label{lem:twin}
Let $G$ be a connected graph, and let $T\subseteq V(G)$ be such that 
all its elements are true twins with one another or all its elements false twins with one another. 
If $|T| \geq 2$, then $|R \cap T| \geq |T| -1$ for all resolving sets $R$ of $G$.
\end{lem}

Notice that if some resolving set contains only $|T|-1$ elements of the set $T$, then it does not matter which twin is left out of the resolving set.
Indeed, the elements of the set $T$ are isomorphic with one another and have the exact same properties.

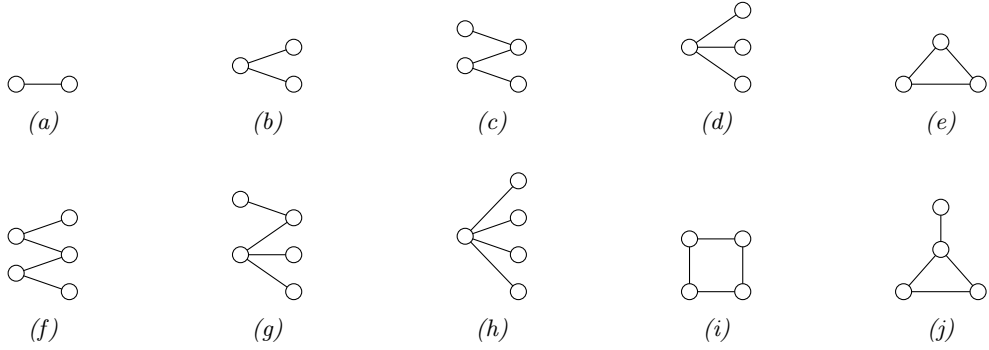
\begin{figure}
\centering
\begin{subfigure}[b]{0.19\linewidth}
\centering
\begin{tikzpicture}[scale=.7]
\draw (0,0) -- (1,0);
\draw \foreach \x in {(0,0),(1,0)} {
 		\x node[circle, draw, fill=white,
                        inner sep=0pt, minimum width=6pt] {}
 	};
\end{tikzpicture}
\caption{ }\label{fig:densecleanK2}
\end{subfigure}
\vspace{3pt}
\begin{subfigure}[b]{0.19\linewidth}
\centering
\begin{tikzpicture}[scale=.7]
\draw (1,0)--(0,.35)--(1,.7);
\draw \foreach \x in {(0,.35),(1,0),(1,.7)} {
 		\x node[circle, draw, fill=white,
                        inner sep=0pt, minimum width=6pt] {}
 	};
\end{tikzpicture}
\caption{ }\label{fig:densecleanK12}
\end{subfigure}
\vspace{3pt}
\begin{subfigure}[b]{0.19\linewidth}
\centering
\begin{tikzpicture}[scale=.7]
\draw (1,0)--(0,.35)--(1,.7)--(0,1.05);
\draw \foreach \x in {(1,0),(0,.35),(1,.7),(0,1.05)} {
 		\x node[circle, draw, fill=white,
                        inner sep=0pt, minimum width=6pt] {}
 	};
\end{tikzpicture}
\caption{ }\label{fig:densecleanP4}
\end{subfigure}
\vspace{3pt}
\begin{subfigure}[b]{0.19\linewidth}
\centering
\begin{tikzpicture}[scale=.7]
\draw \foreach \x in {(1,0),(1,.7),(1,1.4)} {
		(0,.7) -- \x
	};
\draw \foreach \x in {(0,.7),(1,0),(1,.7),(1,1.4)} {
 		\x node[circle, draw, fill=white,
                        inner sep=0pt, minimum width=6pt] {}
 	};
\end{tikzpicture}
\caption{ }\label{fig:densecleanK13}
\end{subfigure}
\vspace{3pt}
\begin{subfigure}[b]{0.19\linewidth}
\centering
\begin{tikzpicture}[scale=.7]
\draw (0,0)--(.7,.8)--(1.4,0)--(0,0);
\draw \foreach \x in {(0,0),(.7,.8),(1.4,0)} {
 		\x node[circle, draw, fill=white,
                        inner sep=0pt, minimum width=6pt] {}
 	};
\end{tikzpicture}
\caption{ }\label{fig:densecleanK3}
\end{subfigure}
\vspace{3pt}
\begin{subfigure}[b]{0.19\linewidth}
\centering
\begin{tikzpicture}[scale=.7]
\draw (1,0)--(0,.35)--(1,.7)--(0,1.05)--(1,1.4);
\draw \foreach \x in {(1,0),(0,.35),(1,.7),(0,1.05),(1,1.4)} {
 		\x node[circle, draw, fill=white,
                        inner sep=0pt, minimum width=6pt] {}
 	};
\end{tikzpicture}
\caption{ }\label{fig:densecleanP5}
\end{subfigure}
\begin{subfigure}[b]{0.19\linewidth}
\centering
\begin{tikzpicture}[scale=.7]
\draw (1,1.4)--(0,1.75);
\draw \foreach \x in {(1,0),(1,.7),(1,1.4)} {
		(0,.7) -- \x
	};
\draw \foreach \x in {(0,.7),(1,0),(1,.7),(1,1.4),(0,1.75)} {
 		\x node[circle, draw, fill=white,
                        inner sep=0pt, minimum width=6pt] {}
 	};
\end{tikzpicture}
\caption{ }\label{fig:densecleanspec1}
\end{subfigure}
\begin{subfigure}[b]{0.19\linewidth}
\centering
\begin{tikzpicture}[scale=.7]
\draw \foreach \x in {(1,0),(1,.7),(1,1.4),(1,2.1)} {
		(0,1.05) -- \x
	};
\draw \foreach \x in {(1,0),(1,.7),(1,1.4),(1,2.1),(0,1.05)} {
 		\x node[circle, draw, fill=white,
                        inner sep=0pt, minimum width=6pt] {}
 	};
\end{tikzpicture}
\caption{ }\label{fig:densecleanK14}
\end{subfigure}
\begin{subfigure}[b]{0.19\linewidth}
\centering
\begin{tikzpicture}[scale=.7]
\draw (0,0)--(1,0)--(1,1)--(0,1)--(0,0);
\draw \foreach \x in {(0,0),(1,0),(1,1),(0,1)} {
 		\x node[circle, draw, fill=white,
                        inner sep=0pt, minimum width=6pt] {}
 	};
\end{tikzpicture}
\caption{ }\label{fig:densecleanC4}
\end{subfigure}
\begin{subfigure}[b]{0.19\linewidth}
\centering
\begin{tikzpicture}[scale=.7]
\draw (0,0)--(.7,.8)--(1.4,0)--(0,0) (.7,.8)--(.7,1.6);
\draw \foreach \x in {(0,0),(.7,.8),(1.4,0),(.7,1.6)} {
 		\x node[circle, draw, fill=white,
                        inner sep=0pt, minimum width=6pt] {}
 	};
\end{tikzpicture}
\caption{ }\label{fig:densecleanspec2}
\end{subfigure}
\caption{Nonisomorphic pieces of edge deletion patterns with which we can construct all nonisomorphic edge deletion patterns with up to 4 edges.}\label{fig:denseclean}
\end{figure}

\begin{lem}\label{lem:denseuniversal}
Let $G$ be a connected graph with a universal vertex $v$.
Then the vertex $v$ is not a basis forced vertex.
\end{lem}
\begin{proof}
Let $R$ be a metric basis of $G$ such that $v \in R$.
Then the set $R \setminus \{v\}$ is not a resolving set.
There exist distinct $x,y \in V(G)$ such that $d(v,x) \neq d(v,y)$ and $d(r,x) = d(r,y)$ for all $r \in R \setminus \{v\}$.
Since $v$ is universal, we have either $x=v$ or $y=v$.
Suppose without loss of generality that $x=v$.
Now, $d(r,y) = 1$ for all $r\in R$, and the vertex $y$ is unique.
Consequently, the set $R[v \leftarrow y]$ is a resolving set.
Thus, the vertex $v$ is not a basis forced vertex.
\end{proof}

We like to describe the graphs we consider by giving the complement graphs, since the complement graph is often far simpler in structure.
The connected components of the complement graph also give us a natural way to partition the vertices.
Consider a graph $G$ such that the complement graph $\cg{G}$ has two connected components: $\cg{G}_1$ and $\cg{G}_2$.
Any two vertices $x \in V(\cg{G}_1)$ and $y \in V(\cg{G}_2)$ are adjacent in $G$.
Thus, $d_G(x,y) = 1$.
We cannot resolve any pair of vertices in $V(\cg{G}_1)$ with a vertex in $V(\cg{G}_2)$ in $G$.
In order to find the metric bases of $G$, we need to make sure that each connected component of $\cg{G}$ contains enough elements of a metric basis to resolve the pairs within that connected component.
In addition to that, we need to be able to resolve pairs of vertices that are in different connected components of $\cg{G}$.

\begin{lem}\label{lem:densecomplete}
Let $G$ be a connected graph.
Let $\cg{G}_i$ be a connected component of $\cg{G}$ such that $\cg{G}_i \simeq K_n$ for some $n \geq 2$.
Then no vertex $v \in V(\cg{G}_i)$ is a basis forced vertex of $G$.
\end{lem}
\begin{proof}
One example of the graph $G$ is illustrated in Figure \ref{fig:denseK4}.

For simplicity, let $i=1$ and let us denote $V(\cg{G}_1) = \{v_1, \ldots , v_n\}$.
Since $\cg{G}_1 \simeq K_n$, we have $v_i \not\sim_G v_j$ for all $i \neq j$ and $v_i \sim_G u$ for all $i$ and $u \in V(G) \setminus V(\cg{G}_i)$.
Thus, $N_G(v_i) = N_G(v_j)$ for all $i \neq j$ and the vertices $v_i$ are false twins in $G$.
According to Lemma \ref{lem:twin}, every resolving set of $G$ contains at least $n-1$ vertices $v_i$.

If there exists a metric basis $R$ such that $v_1 \notin R$, then none of the vertices $v_i$ are basis forced vertices.
Indeed, then the set $R[v_1 \leftarrow v_i]$ is also a metric basis of $G$ as the vertices $v_i$ are twins with each other.

Suppose then that $R$ is a metric basis of $G$ such that $v_i \in R$ for all $i = 1, \ldots , n$. 
We will show that such metric basis cannot in fact exist by proving that the set $R \setminus \{v_i\}$ is a resolving set of $G$ for any $i = 1, \ldots ,n$. 
Let us consider a fixed $v_i$. 
In what follows, the distances we consider are distances in $G$, and we omit the subscript $G$ from the notation of distance. 
Since $d(v_i,u) = 1$ if and only if $u \notin V(\cg{G}_1)$, the vertex $v_i$ can only resolve pairs where one or both vertices are some $v_j$, $j \in \{1, \ldots , n\}$. 
However, all of these pairs can be resolved with other elements of $R$:
\begin{itemize}
\item Consider the vertices $v_j$ and $v_k$ where $j \neq k$.
Suppose without loss of generality that $v_j \neq v_i$.
Now, $v_j \in R \setminus \{v_i\}$ and $d(v_j,v_j) \neq d(v_j,v_k)$.

\item Consider the vertex $v_i$ and some $u \notin V(\cg{G}_1)$.
Let $v_j$ be such that $j \neq i$.
Then $v_j \in R \setminus \{v_i\}$ and $d(v_j,v_i) = 2 \neq 1 = d(v_j,u)$.

\item Consider some vertex $v_j$, where $j \neq i$, and some $u \notin V(\cg{G}_1)$.
Now, $v_j \in R \setminus \{v_i\}$ and $d(v_j,v_j) \neq d(v_j,u)$.
\end{itemize}
In conclusion, the set $R \setminus \{v_i\}$ is a resolving set for any $v_i$.
This contradicts the assumption that the set $R$ is a metric basis of $G$.
Consequently, for every $v_i \in V(\cg{G}_1)$ there exists a metric basis $R$ such that $v_i \notin R$, and thus the vertex $v_i$ is not a basis forced vertex.
\end{proof}

\begin{figure}
\centering
\begin{subfigure}[b]{0.45\linewidth}
\centering
\begin{tikzpicture}[scale=1.1]
	\coordinate (k1) at (0,-1);
	\coordinate (k2) at (0,1);
	\coordinate (k3) at (.3,.2);
	\coordinate (k4) at (-.3,-.2);
	\coordinate (r1) at (2,-.3);
	\coordinate (r2) at (2,.3);
	\coordinate (r3) at (2,-.9);
	\draw (r1) -- (r3);
	\foreach \x in {1,...,4} \foreach \y in {1,...,4} {
		\draw[dotted, thick] (k\x) -- (k\y);
	};
	\foreach \x in {1,...,4} \foreach \y in {1,...,3} {
		\draw (k\x) -- (r\y);
	};
	\draw[dashed] (2,-.3) ellipse (.6cm and 1cm);
	\draw \foreach \x in {(k1),(k2),(k3),(k4),(r1),(r2),(r3)}{
 		\x node[circle, draw, fill=white,
                        inner sep=0pt, minimum width=6pt] {}
 	};
 \end{tikzpicture}
\caption{$\cg{G}_i \simeq K_4$}\label{fig:denseK4}
\end{subfigure}
\begin{subfigure}[b]{0.45\linewidth}
\centering
\begin{tikzpicture}[scale=1.1]
	\coordinate (v0) at (0,1);
	\coordinate (v1) at (0,.1);
	\coordinate (v2) at (.5,-.5);
	\coordinate (v3) at (-.4,-.8);
	\coordinate (r1) at (1.5,-1);
	\coordinate (r2) at (1.7,-.5);
	\foreach \x in {1,...,3} {
		\draw[dotted, thick] (v0) -- (v\x);
	};
	\draw (v1) -- (v2) -- (v3) -- (v1);
	\foreach \x in {0,...,3} \foreach \y in {1,...,2} {
		\draw (v\x) -- (r\y);
	};
	\draw[dashed] (1.6,-.75) ellipse (.5cm and .7cm);
	\draw \foreach \x in {(v0),(v1),(v2),(v3),(r1),(r2)}{
 		\x node[circle, draw, fill=white,
                        inner sep=0pt, minimum width=6pt] {}
 	};
 \end{tikzpicture}
\caption{$\cg{G}_i \simeq K_{1,3}$}\label{fig:denseK13}
\end{subfigure}
\caption{Example illustrations of the graphs considered in Lemmas \ref{lem:densecomplete} and \ref{lem:densestar}. The graphs may have any (nonempty) structure inside the dashed area.}
\end{figure}
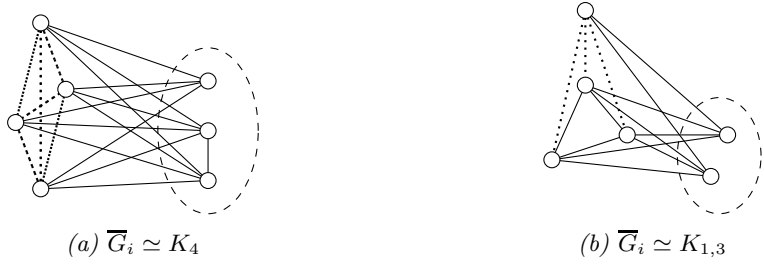

\begin{lem}\label{lem:densestar}
Let $G$ be a connected graph.
Let $\cg{G}_i$ be a connected component of $\cg{G}$ such that $\cg{G}_i \simeq K_{1,n}$ for some $n \geq 2$.
Then no vertex $v \in V(\cg{G}_i)$ is a basis forced vertex of $G$.
\end{lem}
\begin{proof}
One example of the graph $G$ is illustrated in Figure \ref{fig:denseK13}.

Let $\cg{G}_1 \simeq K_{1,n}$ for some $n \geq 2$, and denote $V(\cg{G}_1) = \{v_0, \ldots ,v_n\}$, where $\deg_{\cg{G}} (v_0) = n$ and $\deg_{\cg{G}} (v_i) = 1$ for all $i \neq 0$.
Since $G$ is connected, $V(G) \setminus V(\cg{G}_1) \neq \emptyset$.
The vertices $v_i$ where $i \neq 0$ are true twins in $G$.
According to Lemma \ref{lem:twin} any resolving set of $G$ contains at least $n-1$ vertices $v_i$, $i \neq 0$.

We will first show that no $v_i$, where $i\neq 0$, is a basis forced vertex of $G$.
Suppose that $R$ is a resolving set of $G$ such that $v_i \in R$ for all $i \neq 0$.
Consider a fixed $v_i$.
Let us show that the set $R[v_i \leftarrow v_0]$ is a resolving set of $G$.
In what follows, all distances that we consider are distances in $G$.
Consider the pairs of vertices that $v_i$ resolves.
Let $x,y \in V(G)$ be such that $d(v_i,x) \neq d(v_i,y)$.
We have $d(v_i,u) = 1$ if and only if $u \neq v_i$ and $u \neq v_0$.
Consequently, $x \in \{v_i,v_0\}$ or $y \in \{v_i,v_0\}$.
Suppose that $x = v_0$.
Then $v_0$ resolves $x$ and $y$ no matter what $y$ is.
Suppose then that $x = v_i$.
If $d(v_0,y) \neq d(v_0,v_i)$, then we are done.
Suppose that $d(v_0,y) = d(v_0,v_i) = 2$.
Since $d(v_0,u) = 2$ if and only if $u=v_j$ for some $j = 1, \ldots , n$, we have $y = v_j$ for some $j = 1, \ldots, n$.
However, we also have that $y \neq v_i$, and thus $y \in R[v_i \leftarrow v_0]$ and $y$ itself resolves $v_i$ and $y$.
Thus, $R[v_i \leftarrow v_0]$ is a resolving set of $G$ and $v_i$ is not a basis forced vertex of $G$.

Let us then show that $v_0$ is not a basis forced vertex.
Let $R$ be a metric basis of $G$ such that $v_0 \in R$.
Due to what we have already shown above, we may assume that $v_i \notin R$ for some $i \in \{1, \ldots , n\}$ and $v_j \in R$ for all $j \neq i$.
Let us show that the set $R [v_0 \leftarrow v_i]$ is a metric basis of $G$.

Since $R$ is a metric basis, the set $R \setminus \{v_0\}$ is not a resolving set.
Suppose that the pair $x,y \in V(G)$ is not resolved by any element of $R \setminus \{v_0\}$.
We will show that then either $x=v_i$ or $y=v_i$.
It is clear that $x,y \notin R \setminus \{v_0\}$.
If $x,y \notin V(\cg{G}_1)$, then $d(v_0,x) = d(v_0,y) = 1$ and the pair $x,y$ is not resolved by any element of $R$, which contradicts the assumption that the set $R$ is a metric basis of $G$.
For any $j \neq i$ we have $d(v_j,v_0) = 2$ and $d(v_j,u) = 1$ for all $u \in V(G) \setminus \{v_0\}$.
If $x = v_0$ or $y = v_0$, then $v_j$ (where $j \neq i$) resolves $x$ and $y$, which contradicts the assumption that no element of $R \setminus \{v_0\}$ resolves $x$ and $y$, since we have $v_j \in R$ for all $j \neq i$.
Thus, we have $x \neq v_0$ and $y \neq v_0$, and the only option we now have left is that either $x = v_i$ or $y = v_i$.
Consequently, the set $R[v_0 \leftarrow v_i]$ is a metric basis of $G$ and $v_0$ is not a basis forced vertex of $G$.
\end{proof}

\begin{lem}\label{lem:densepath4}
Let $G$ be a connected graph with at least 5 vertices.
Let $\cg{G}_i$ be a connected component of $\cg{G}$ such that $\cg{G}_i \simeq P_4$.
Then the vertices $v \in V(\cg{G}_i)$ are not basis forced vertices of $G$.
\end{lem}
\begin{proof}
One example of the graph $G$ is illustrated in Figure \ref{fig:denseP4}.

Since the path $P_4$ is self-complementary, we have $G[V(\cg{G}_i)] \simeq P_4$.
Let us denote the elements of $V(\cg{G}_i)$ by $v_i$ so that $v_1 v_2 v_3 v_4$ is a path in $G$.
Since $G$ has at least 5 vertices, $V(G) \setminus V(\cg{G}_i) \neq \emptyset$.
Consequently, $d(v_1,v_4) = 2$.

In order to resolve $v_i$ and $v_j$, where $i \neq j$, we need at least two elements of $V(\cg{G}_i)$ in any resolving set.
Indeed, we have $d(u,v_i) = 1 = d(u,v_j)$ for all $u \in V(G) \setminus V(\cg{G}_i)$, so the only vertices that possibly resolve $v_i$ and $v_j$ are the elements of $V(\cg{G}_i)$.
However, for any element $v \in V(\cg{G}_i)$ there are at least two vertices $v_i$ at the same distance from it.
For example, we have $d(v_1,v_3) = 2 = d(v_1,v_4)$ and $d(v_2,v_1) = 1 = d(v_2,v_3)$.
Thus, any resolving set of $G$ contains at least two elements of $V(\cg{G}_i)$.

Let us then show that two elements of $V(\cg{G}_i)$ are enough.
The elements of $V(\cg{G}_i)$ can only resolve pairs of vertices where one or both vertices are in $V(\cg{G}_i)$.
Let us show that all such pairs are resolved by $v_1$ and $v_2$.
Consider distinct $x,y \in V(G)$.
Suppose that $x,y \in V(\cg{G}_i)$.
If $x=v_1$ or $x=v_2$, then $x$ and $y$ are resolved by $x$ itself (similarly for $y$).
Let $x=v_3$ and $y=v_4$.
Now, $d(v_2,v_3) = 1 \neq 2 = d(v_2,v_4)$, and thus $x$ and $y$ are resolved.
Suppose then that $x \in V(\cg{G}_i)$ and $y \in V(G) \setminus V(\cg{G}_i)$.
Now, $d(v_1,y) = 1 = d(v_2,y)$.
However, $d(v_1,x) \neq 1$ when $x \neq v_2$ and $d(v_2,x) \neq 1$ when $x=v_2$.
In conclusion, it is enough for a resolving set to contain two elements of $V(\cg{G}_i)$; $v_1$ and $v_2$, or $v_3$ and $v_4$ by symmetry.
\end{proof}

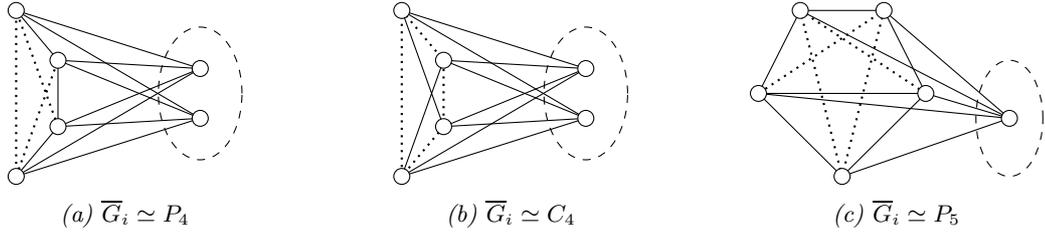
\begin{figure}
\centering
\begin{subfigure}[b]{0.32\linewidth}
\centering
\begin{tikzpicture}[scale=1.1]
	\coordinate (p1) at (.3,-.4);
	\coordinate (p2) at (-.2,1);
	\coordinate (p3) at (-.2,-1);
	\coordinate (p4) at (.3,.4);
	\coordinate (r1) at (2,-.3);
	\coordinate (r2) at (2,.3);
	\draw[dotted, thick] (p1) -- (p2) -- (p3) -- (p4);
	\draw (p1) -- (p3)
		  (p1) -- (p4)
		  (p2) -- (p4);
	\foreach \x in {1,...,4} \foreach \y in {1,...,2} {
		\draw (p\x) -- (r\y);
	};
	\draw[dashed] (2,0) ellipse (.5cm and .8cm);
	\draw \foreach \x in {(p1),(p2),(p3),(p4),(r1),(r2)}{
 		\x node[circle, draw, fill=white,
                        inner sep=0pt, minimum width=6pt] {}
 	};
 \end{tikzpicture}
\caption{$\cg{G}_i \simeq P_4$}\label{fig:denseP4}
\end{subfigure}
\hfill
\begin{subfigure}[b]{0.32\linewidth}
\centering
\begin{tikzpicture}[scale=1.1]
	\coordinate (p1) at (.3,-.4);
	\coordinate (p2) at (-.2,1);
	\coordinate (p3) at (-.2,-1);
	\coordinate (p4) at (.3,.4);
	\coordinate (r1) at (2,-.3);
	\coordinate (r2) at (2,.3);
	\draw[dotted, thick] (p1) -- (p3) -- (p2) -- (p4) -- (p1);
	\draw (p1) -- (p2) (p3) -- (p4);
	\foreach \x in {1,...,4} \foreach \y in {1,...,2} {
		\draw (p\x) -- (r\y);
	};
	\draw[dashed] (2,0) ellipse (.5cm and .8cm);
	\draw \foreach \x in {(p1),(p2),(p3),(p4),(r1),(r2)}{
 		\x node[circle, draw, fill=white,
                        inner sep=0pt, minimum width=6pt] {}
 	};
 \end{tikzpicture}
\caption{$\cg{G}_i \simeq C_4$}\label{fig:denseC4}
\end{subfigure}
\hfill
\begin{subfigure}[b]{0.32\linewidth}
\centering
\begin{tikzpicture}[scale=1.1]
	\coordinate (p1) at (1,0);
	\coordinate (p2) at (-.5,1);
	\coordinate (p3) at (0,-1);
	\coordinate (p4) at (.5,1);
	\coordinate (p5) at (-1,0);
	\coordinate (r1) at (2,-.3);
	\draw[dotted, thick] (p1) -- (p2) -- (p3) -- (p4) -- (p5);
	\draw (p1) -- (p3) -- (p5) -- (p1) -- (p4) -- (p2) -- (p5);
	\foreach \x in {1,...,5} {
		\draw (p\x) -- (r1);
	};
	\draw[dashed] (r1) ellipse (.4cm and .7cm);
	\draw \foreach \x in {(p1),(p2),(p3),(p4),(p5),(r1)}{
 		\x node[circle, draw, fill=white,
                        inner sep=0pt, minimum width=6pt] {}
 	};
 \end{tikzpicture}
\caption{$\cg{G}_i \simeq P_5$}\label{fig:denseP5}
\end{subfigure}
\caption{Example illustrations of the graphs considered in Lemmas \ref{lem:densepath4}, \ref{lem:densecycle4} and \ref{lem:denseP5}. The graphs may have any (nonempty) structure inside the dashed area.}
\end{figure}

\begin{lem}\label{lem:densecycle4}
Let $G$ be a connected graph with at least 5 vertices.
Let $\cg{G}_i$ be a connected component of $\cg{G}$ such that $\cg{G}_i \simeq C_4$.
Then the vertices $v \in V(\cg{G}_i)$ are not basis forced vertices of $G$.
\end{lem}
\begin{proof}
Let us denote the vertices of $\cg{G}_i$ by $v_1, \ldots , v_4$ so that $v_1 v_2 v_3 v_4 v_1$ is a cycle in $\cg{G}$.
Then $v_1$ and $v_3$ are adjacent in $G$, and $v_2$ and $v_4$ are adjacent in $G$ (see Figure \ref{fig:denseC4}).
The vertices $v_1$ and $v_3$ are true twins, and so are $v_2$ and $v_4$.
According to Lemma \ref{lem:twin} any resolving set contains $v_1$ or $v_3$ (and $v_2$ or $v_4$). 

Let us show that a metric basis of $G$ contains at most two vertices of $V(\cg{G}_i)$.
No element of $V(\cg{G}_i)$ can resolve $x$ and $y$ if both $x \notin V(\cg{G}_i)$ and $y \notin V(\cg{G}_i)$.
Thus, it is sufficient to consider distinct $x,y \in V(G)$ such that $x \in V(\cg{G}_i)$.
Let us show that we can resolve any such pair $x,y$ with $v_1$ and $v_2$.
Clearly, if $x \in \{v_1,v_2\}$ or $y \in \{v_1,v_2\}$, then $x$ and $y$ are resolved.
Suppose then that $x = v_3$.
Now, $d(v_2,x) = 2$.
The only vertices that are at distance 2 from $v_2$ are $v_1$ and $v_3$.
Thus, if we had $d(v_2,x) = d(v_2,y)$, then $y = v_1$ and $d(v_1,y) \neq d(v_1,x)$.
Similarly, if $x=v_4$ and $d(v_1,x) = 2 = d(v_1,y)$, then $y = v_2$ and $d(v_2,y) \neq d(v_2,x)$.
In conclusion, there exists a metric basis $R$ of $G$ such that $R \cap V(\cg{G}_i) = \{v_1,v_2\}$.
By symmetry, there also exists a metric basis $R$ of $G$ such that $R \cap V(\cg{G}_i) = \{v_3,v_4\}$.
Thus, none of the elements of $V(\cg{G}_i)$ are basis forced vertices.
\end{proof}

The following lemma is used in this section for $k=2$.
We will, however, use the following lemma for other values of $k$ later in Section \ref{sec:Colour}.

However, this more general form will be used later in Section \ref{sec:Colour}.

\begin{lem}\label{lem:denseP5}
Let $k$ be an even positive integer.
Let $G$ be a connected graph such that $\cg{G} \simeq P_5 \cup \cdots \cup P_5 \cup \cg{K}_m$ where $P_5$ appears $\frac{k}{2}$ times and $m \geq 1$.
The graph $G$ has $k$ basis forced vertices and $\dim (G) = k+m-1$.
\end{lem}
\begin{proof}
The graph $G$ for $k=2$ and $m=1$ is illustrated in Figure \ref{fig:denseP5}.

Let us denote the copies of $P_5$ by $P_5^i$ where $i=1,\ldots , \frac{k}{2}$.
We denote the vertices of each $P_5^i$ by $v_1^i, \ldots , v_5^i$ where the vertices are numbered in the order natural for paths.
We further denote the vertices of $\cg{K}_m$ by $u_j$.

It is clear that in order to resolve all pairs $x,y \in V(P_5^i)$ in $G$, we need at least two elements of $V(P_5^i)$.
The vertices $u_j$ are twins with one another.
Thus, according to Lemma \ref{lem:twin}, there exists at most one $u_j$ such that $u_j \notin R$ for any resolving set $R$.
Therefore, $\dim (G) \geq 2 \cdot \frac{k}{2} +m-1 = k+m-1$.

Let us consider how we can resolve the elements of $V(P_5^i)$ in $G$. 
It is easy to verify that the only two element subsets of $V(P_5^i)$ that resolve $V(P_5^i)$ are $\{v_1^i,v_3^i\}$, $\{v_1^i,v_5^i\}$, $\{v_2^i,v_3^i\}$, $\{v_2^i,v_4^i\}$, $\{v_3^i,v_4^i\}$ and $\{v_3^i,v_5^i\}$.

Let $R \subseteq V(G)$ be such that $R \cap V(P_5^i) = \{v_2^i,v_4^i\}$ for all $i$ and $R \cap V(\cg{K}_m) = V(\cg{K}_m) \setminus \{u_j\}$ for some $j$.
Let us show that $R$ is a resolving set of $G$.
To that end, consider distinct vertices $x,y \in V(G)$.
If $x,y \in V(P_5^i)$ for some $i$, then the set $R$ resolves $x$ and $y$ by the observation above.
Suppose that $x \in V(P_5^i)$ and $y \notin V(P_5^i)$ for some $i$.
Now, $d(v_2^i,y) = d(v_4^i,y) = 1$.
One of the distances $d(v_2^i,x)$ and $d(v_4^i,x)$ is 0 or 2 for any $x \in V(P_5^i)$.
Thus, we have $d(v_2^i,x) \neq d(v_2^i,y)$ or $d(v_4^i,x) \neq d(v_4^i,y)$.
Suppose finally that $x,y \notin V(P_5^i)$ for all $i$.
In other words, $x,y \in V(\cg{K}_m)$.
Since $u_j \notin R$ for only one vertex $u_j \in V(\cg{K}_m)$, the vertices $x$ and $y$ are resolved by some element of $R$.
Therefore, the set $R$ is a resolving set of $G$.
Since $|R| = 2 \cdot \frac{k}{2} + m-1$, the set $R$ is a metric basis of $G$ and $\dim (G) = k+m-1$.

Suppose then that $R$ is a resolving set of $G$.
Let us denote $S=R \cap V(P_5^i)$.
Suppose that $|S|=2$ and $S \neq \{v_2^i,v_4^i\}$ for some $i$.
(In other words, the intersection $R \cap V(P_5^i)$ is one of the five listed earlier that are not $\{v_2^i,v_4^i\}$.) 
It is easy to verify (by e.g. making a $5 \times 5$-table of distances) that there exists a vertex $x \in V(P_5^i)$ such that $d(s,x) = 1$ for all $s\in S$ if and only if $S \neq \{v_2^i,v_4^i\}$. 
Thus, since $S \neq \{v_2^i,v_4^i\}$, we have $d(s,x) = d(s,u_j)$ for all $s\in S$ and $u_j \in V(\cg{K}_m)$. 
Now $u_j \in R$ for all $j$, since $R$ is a resolving set of $G$ and $d(u_l, u_j) = 1 = d(u_l,x)$ for all $l\neq j$. 
We have $|R| \geq 2 \cdot \frac{k}{2} + m = k + m$, since $|R \cap V(P_5^j)| \geq 2$ for all $j \in \{1, \ldots , \frac{k}{2} \}$ as we pointed out earlier. 
Consequently, $R$ is not a metric basis of $G$.

In conclusion, a set $R \subseteq V(G)$ is a metric basis of $G$ if and only if $R \cap V(P_5^i) = \{v_2^i,v_4^i\}$ for all $i = 1, \ldots , \frac{k}{2}$ and $R \cap V(\cg{K}_m) = V(\cg{K}_m) \setminus \{u_j\}$ for some $j \in \{1, \ldots , m\}$.
Consequently, each $v_2^i$ and $v_4^i$ is a basis forced vertex of $G$ and $G$ has $k$ basis forced vertices in total.
\end{proof}

Now we are ready for the main result in this section.

\begin{thm}\label{thm:denseedgebound}
Let $G$ be a connected graph with $n \geq 6$ vertices and at least one basis forced vertex.
Then 
\[ |E(G)| \leq \frac{n(n-1)}{2} - 4. \]
If $|E(G)| = \frac{n(n-1)}{2} - 4$, then $\cg{G} \simeq P_5 \cup \cg{K}_{n-5}$ and $G$ has two basis forced vertices.
\end{thm}
\begin{proof}
Let us consider $G$ as constructed from $K_n$ by removing edges. 

Suppose that $G$ is obtained from $K_n$ by removing up to three edges.
Then the graph $\cg{G}$ can be presented as a union of some of the following graphs (see Figure \ref{fig:denseclean} (a-e)): $K_2$, $K_{1,2}$, $P_4$, $K_{1,3}$, $K_3$ and $\cg{K}_{n-k}$ for some $k$.
Lemmas \ref{lem:denseuniversal}, \ref{lem:densecomplete}, \ref{lem:densestar} and \ref{lem:densepath4} cover all possible cases, and in each case the graph $G$ does not have any basis forced vertices.

We have now shown that if $|E(G)| \geq \frac{n(n-1)}{2} - 3$, then the graph $G$ does not have any basis forced vertices.
Consequently, if $G$ has basis forced vertices, then $|E(G)| \leq \frac{n(n-1)}{2} - 4$.

Let us then prove the latter claim.
Let $G$ be a graph that is obtained from $K_n$ by removing four edges.
Suppose that $\cg{G}$ can be presented as the union of some of the following graphs (see Figure \ref{fig:denseclean}(a-e,h,i)): $K_2$, $K_{1,2}$, $P_4$, $K_{1,3}$, $K_3$, $K_{1,4}$, $C_4$ and $\cg{K}_{n-k}$ for some $k$.
As before, we can use Lemmas \ref{lem:denseuniversal}, \ref{lem:densecomplete}, \ref{lem:densestar}, \ref{lem:densepath4} and now \ref{lem:densecycle4} also, and obtain that in all cases the graph $G$ does not have any basis forced vertices.
Thus, if $G$ has basis forced vertices, then there are only three possibilities for the edge deletion pattern: the graphs in Figure \ref{fig:denseclean}(f,g,j).
Let us consider each case separately.

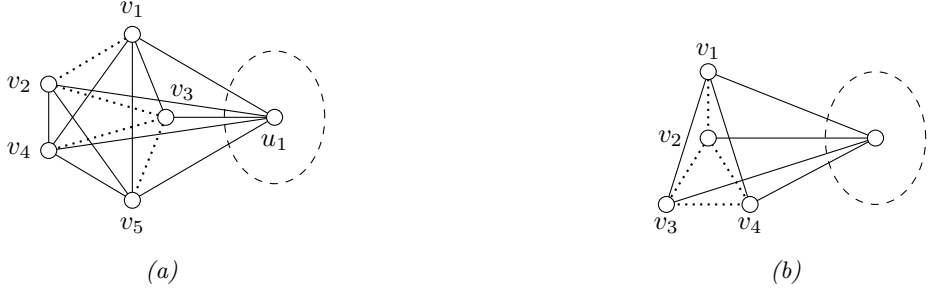
\begin{figure}
\centering
\begin{subfigure}[b]{0.45\linewidth}
\centering
\begin{tikzpicture}[scale=1.1]
	\coordinate (v1) at (.3,1);
	\coordinate (v2) at (-.7,.4);
	\coordinate (v3) at (.7,0);
	\coordinate (v4) at (-.7,-.4);
	\coordinate (v5) at (.3,-1);
	\coordinate (r1) at (2,0);
	\draw[dotted, thick] (v1) -- (v2) -- (v3) -- (v4) (v3) -- (v5);
	\draw (v4) -- (v1) -- (v3)
		  (v1) -- (v5) -- (v2) -- (v4) -- (v5);
	\foreach \x in {1,...,5} {
		\draw (v\x) -- (r1);
	};
	\draw[dashed] (2,0) ellipse (.6cm and .8cm);
	\draw \foreach \x in {(v1),(v2),(v3),(v4),(v5),(r1)}{
 		\x node[circle, draw, fill=white,
                        inner sep=0pt, minimum width=6pt] {}
 	};
 	\draw {
 		(.3,1.3) node[] {$v_1$}
 		(-1.05,.4) node[] {$v_2$}
 		(.9,.3) node[] {$v_3$}
 		(-1.05,-.4) node[] {$v_4$}
 		(.3,-1.3) node[] {$v_5$}
 		(2,-.3) node[] {$u_1$}
 	};
 \end{tikzpicture}
\caption{ }\label{fig:densespecial1}
\end{subfigure}
\hfill
\begin{subfigure}[b]{0.45\linewidth}
\centering
\begin{tikzpicture}[scale=1.1]
	\coordinate (v1) at (0,.8);
	\coordinate (v2) at (0,0);
	\coordinate (v3) at (-.5,-.8);
	\coordinate (v4) at (.5,-.8);
	\coordinate (r1) at (2,0);
	\draw[dotted, thick] (v1) -- (v2) -- (v3) -- (v4) -- (v2);
	\draw (v4) -- (v1) -- (v3);
	\foreach \x in {1,...,4} {
		\draw (v\x) -- (r1);
	};
	\draw[dashed] (2,0) ellipse (.6cm and .8cm);
	\draw \foreach \x in {(v1),(v2),(v3),(v4),(r1)}{
 		\x node[circle, draw, fill=white,
                        inner sep=0pt, minimum width=6pt] {}
 	};
 	\draw {
 		(0,1.05) node[] {$v_1$}
 		(-.45,0) node[] {$v_2$}
 		(-.5,-1.05) node[] {$v_3$}
 		(.5,-1.05) node[] {$v_4$}
 	};
 \end{tikzpicture}
\caption{ }\label{fig:densespecial2}
\end{subfigure}
\caption{Example illustrations of the two special cases of the proof of Theorem \ref{thm:denseedgebound}. The graphs may have any (nonempty) structure inside the dashed area.}
\end{figure}

Suppose that $\cg{G}$ is the disjoint union of the graph in Figure \ref{fig:densecleanspec1} and $\cg{K}_{n-5}$.
Let us denote the elements as in Figure \ref{fig:densespecial1}, that is, the isolated vertices are denoted by $u_i$ and the five other vertices are denoted by $v_i$. 
It is easy to verify that the sets $V(G) \setminus \{ v_2, v_3, v_5\}$ and $V(G) \setminus \{ v_1, v_4, u_i\}$ are metric bases for any $i$.
Thus, the graph $G$ does not have any basis forced vertices.

Suppose that $\cg{G}$ is the disjoint union of the graph in Figure \ref{fig:densecleanspec2} and $\cg{K}_{n-4}$.
Let us denote the elements of $V(\cg{G}_i)$ by $v_1,\ldots,v_4$ as in Figure \ref{fig:densespecial2}. 
It is clear that in order to resolve  $V(\cg{G}_i)$ in $G$, we need at least two elements of $V(\cg{G}_i)$.
It is easy to verify that $\{v_1,v_3\}$ and $\{v_2,v_4\}$ do the job.
Both of these options also resolve any pair $x \in V(\cg{G}_i)$ and $y \in V(G) \setminus V(\cg{G}_i)$.
Indeed, for any $x \in V(\cg{G}_i)$ one of the distances $d(x,v_1)$ and $d(x,v_3)$ is 2 or 0 (and similarly for $v_2$ and $v_4$). 
Thus, the graph $G$ does not have any basis forced vertices.

Thus, if $G$ has basis forced vertices, then $\cg{G} \simeq P_5 \cup \cg{K}_{n-5}$.
The graph $G$ is illustrated in Figure \ref{fig:dense2bforced}.
According to Lemma \ref{lem:denseP5}, this graph does indeed have exactly two basis forced vertices.
\end{proof}

It is a bit odd that the densest graphs that have basis forced vertices have two of them.
If a graph $G$ has one basis forced vertex, then $|E(G)| \leq \frac{n(n-1)}{2} - 5$.
A graph that attains this bound is illustrated in Figure \ref{fig:dense1bforced}.
The black vertex is the basis forced vertex.
The metric dimension of this graph is 2.
It has two metric bases; in addition to the black vertex we must include either of the gray vertices.

If a graph has three or more basis forced vertices, then the densest graph that we have found has $\frac{n(n-1)}{2} - 7$ edges.
It seems that in general the bound in Theorem \ref{thm:denseedgebound} leaves room for improvement.
Indeed, in Theorem \ref{thm:genedgebound} we will show a more general bound on the number of edges that depends also on the number of basis forced vertices that the graph contains.

\begin{figure}[ht!]
\begin{subfigure}[b]{0.45\linewidth}
\centering
\begin{tikzpicture}[scale=.6]
	\draw (0,2) -- (1,3) -- (1,1) -- (0,2) -- (-1,3) -- (-1,1) -- (0,2) -- (0,0) -- (1,1) -- (-1,1) -- (0,0) (-1,3) -- (1,3);
 	\draw \foreach \x in {(0,0),(-1,1),(1,1),(0,2)} {
 		\x node[circle, draw, fill=white, inner sep=0pt, minimum width=6pt] {}
 	};
 	\draw \foreach \x in {(-1,3),(1,3)} {
 		\x node[circle, draw, fill=black, inner sep=0pt, minimum width=6pt] {}
 	};
 \end{tikzpicture}
 \caption{The unique graph with 6 vertices that attains the bound in Theorem \ref{thm:denseedgebound}}\label{fig:dense2bforced}
\end{subfigure}
\hfill
\begin{subfigure}[b]{0.45\linewidth}
\centering
\begin{tikzpicture}[scale=.6]
	\draw (0,2) -- (1,3) -- (1,1) -- (0,2) -- (-1,3) -- (-1,1) -- (0,2) -- (0,0) -- (1,1) -- (-1,1) (1,3) -- (-1,3);
 	\draw \foreach \x in {(0,0),(-1,1),(1,1),(0,2)} {
 		\x node[circle, draw, fill=white, inner sep=0pt, minimum width=6pt] {}
 	};
 	\draw \foreach \x in {(-1,3)} {
 		\x node[circle, draw, fill=black, inner sep=0pt, minimum width=6pt] {}
 	};
 	\draw \foreach \x in {(1,3),(-1,1)} {
 		\x node[circle, draw, fill=gray!75, inner sep=0pt, minimum width=6pt] {}
 	};
 \end{tikzpicture}
 \caption{A graph with one basis forced vertex and $\frac{n(n-1)}{2} -5$ edges.}\label{fig:dense1bforced}
\end{subfigure}
\caption{ }
\end{figure}

\section{The Colour Graph $G_R$}\label{sec:Colour}

Let $G$ be a graph and $R \subseteq V(G)$.
Let $r \in R$.
We denote 
\[
\U_R (r) = \{ \{x,y\} \in V(G)^2 \ | \ d(r,x) \neq d(r,y) \text{ and } \forall \, t \in R \setminus \{r\} \colon d(t,x) = d(t,y) \}.
\]
The set $\U_R (r)$ consists of the pairs of vertices for which $r$ is the unique element in $R$ that resolves the pairs.

We denote by $G_R$ the graph with the same vertex set as $G$ and the edge set
\[
\bigcup\limits_{r\in R} \U_R (r).
\]
Each $r \in R$ is assigned a colour, and we colour the edges in $G_R$ given by $\U_R (r)$ with the colour associated with $r$.

\begin{ex}
Consider again the familiar graph $G$ illustrated in Figure \ref{fig:colourExG}.
Let us construct the graph $G_R$ with respect to the resolving set $R=\{r_1,r_2\}$.
We have $\U_R (r_1) = \{\{r_1,v_1\}, \{r_1,v_3\},\{v_1,v_3\}\},\{v_2,v_4\}\}$ and $\U_R (r_2) = \{\{r_2,v_1\}, \{r_2,v_2\},\{v_1,v_2\}\},\{v_3,v_4\}\}$.
When we assign black to $r_1$ and 'dashed' to $r_2$, we can visualise the graph $G_R$ as in Figure \ref{fig:colourExGR}.
\end{ex}

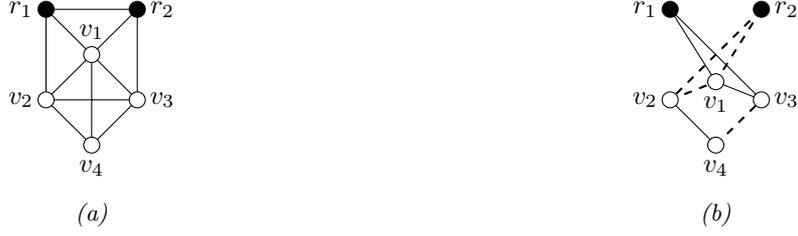
\begin{figure}
\centering
\begin{subfigure}[b]{0.45\linewidth}
\centering
\begin{tikzpicture}[scale=.6]
	\draw (0,2) -- (1,3) -- (1,1) -- (0,2) -- (-1,3) -- (-1,1) -- (0,2) -- (0,0) -- (1,1) -- (-1,1) -- (0,0) (-1,3) -- (1,3);
 	\draw \foreach \x in {(0,0),(-1,1),(1,1),(0,2)} {
 		\x node[circle, draw, fill=white, inner sep=0pt, minimum width=6pt] {}
 	};
 	\draw \foreach \x in {(-1,3),(1,3)} {
 		\x node[circle, draw, fill=black, inner sep=0pt, minimum width=6pt] {}
 	};
	\draw {
		(-1.55,3) node[] {$r_1$}
		(1.55,3) node[] {$r_2$}
		(0,-.5) node[] {$v_4$}
		(-1.55,1) node[] {$v_2$}
		(1.55,1) node[] {$v_3$}
		(0,2.5) node[] {$v_1$}
	};
 \end{tikzpicture}
 \caption{ }\label{fig:colourExG}
\end{subfigure}
\hfill
\begin{subfigure}[b]{0.45\linewidth}
\centering
\begin{tikzpicture}[scale=.6]
	\draw (-1,3) -- (1,1) -- (0,1.4) -- (-1,3) (0,0) -- (-1,1);
	\draw[dashed,thick] (1,3) -- (-1,1) -- (0,1.4) -- (1,3) (0,0) -- (1,1);
 	\draw \foreach \x in {(0,0),(-1,1),(1,1),(0,1.4)} {
 		\x node[circle, draw, fill=white, inner sep=0pt, minimum width=6pt] {}
 	};
 	\draw \foreach \x in {(-1,3),(1,3)} {
 		\x node[circle, draw, fill=black, inner sep=0pt, minimum width=6pt] {}
 	};
	\draw {
		(-1.55,3) node[] {$r_1$}
		(1.55,3) node[] {$r_2$}
		(0,-.5) node[] {$v_4$}
		(-1.55,1) node[] {$v_2$}
		(1.55,1) node[] {$v_3$}
		(0,.9) node[] {$v_1$}
	};
 \end{tikzpicture}
 \caption{ }\label{fig:colourExGR}
\end{subfigure}
\caption{The graph $G$ with the resolving set $R = \{r_1,r_2\}$ and the colour graph $G_R$.}\label{fig:colourEx}
\end{figure}

Let us explore some basic properties of $G_R$.
If there is no edge between $x$ and $y$ in $G_R$, then either $R$ does not resolve $x$ and $y$ or there are at least two elements in $R$ that resolve $x$ and $y$.
If $R$ is a resolving set, then only the latter is possible. 

Since both $r\in R$ and $s \in R$ resolve the pair $r,s$, the edge $\{r,s\}$ is not present in $G_R$.
Consequently, the set $R$ is independent in $G_R$.
Moreover, if there does exist an edge incident to $r\in R$ in $G_R$, then that edge has the colour associated with $r$, since $r$ resolves any pair where $r$ itself is included.

If the colour associated with some $r\in R$ is not present in $G_R$, then any pair that $r$ resolves is resolved by some other element of $R$.
Thus, the set $R \setminus \{r\}$ is a resolving set of $G$.
Consequently, if $R$ is a metric basis of $G$, then the graph $G_R$ has at least one edge of the colour associated with each $r \in R$.
In other words, the set $\U_R (r)$ is nonempty for all $r\in R$.

\begin{lem}\label{lem:colourprops}
Let $G$ be a graph.
The following properties hold.
\begin{enumerate}
\item[(i)] Any colour that appears in a cycle of $G_R$ appears at least twice in that cycle.

\item[(ii)] Let $R$ be a resolving set of $G$ and $x,y,z \in V(G)$.
If the edges $\{x,y\}$ and $\{x,z\}$ have the same colour in $G_R$, then the edge $\{y,z\}$ also has the same colour in $G_R$.

\item[(iii)] If $b \in V(G)$ is a basis forced vertex of $G$ and $R$ is a metric basis of $G$, then the graph $G_R$ has at least two edges of the colour associated with $b$.

\item[(iv)] If $b \in V(G)$ is a basis forced vertex of $G$ and $R$ is a metric basis of $G$, then the graph $G_R$ has at least one edge $\{x,y\}$, $x,y \in V \setminus R$, of the colour associated with $b$.
\end{enumerate}
\end{lem}
\begin{proof}
(i) Suppose to the contrary that the edge $\{v_1,v_k\}$ is the only edge of its colour (associated with $r\in R$) in the cycle $v_1 v_2 v_3 \ldots v_k v_1$, $k\geq 3$.
Since the colour of the edge $\{v_i,v_{i+1}\}$ where $i \in \{1, \ldots , k-1 \}$ is not the colour associated with $r$, we have $d(r,v_i)=d(r,v_{i+1})$ for all $i \in \{1,\ldots,k-1\}$.
Consequently, $d(r,v_1) = d(r,v_k)$ and $\{v_1,v_k\} \notin \U_R (r)$, a contradiction.

(ii) For all $s \in R \setminus \{r\}$ we have $d(s,y) = d(s,x) = d(s,z)$.
As $R$ is a resolving set of $G$, we have $d(r,y) \neq d(r,z)$.
Thus, $\{y,z\} \in \U_R (r)$.

(iii) Suppose $\U_R (b) = \{\{u,v\}\}$.
Since every pair $r,s \in R$ is resolved by both $r$ and $s$, we can assume that $v \notin R$.
The vertex $v$ resolves the pair $\{u,v\}$ and all other pairs are resolved by the elements of $R \setminus \{b\}$.
Thus, the set $R[b \leftarrow v]$ is a metric basis of $G$, a contradiction.

(iv) Suppose to the contrary that no such edge exists.
According to (iii) the graph $G_R$ has at least two edges of the colour associated with $b$.
Let these edges be $\{x_1,x_2\}$ and $\{y_1,y_2\}$.
Due to our assumption, both of these edges have at least one endpoint in $R$.
If there is an edge incident to $r \in R$ in $G_R$, then that edge has the colour associated with $r$, since $r$ resolves any pair $r,x$, where $x \in V(G)$.
Thus, we may assume that $x_1 = b = y_1$.
Since every pair $r,s \in R$ is resolved by both $r$ and $s$, we have $x_2 \notin R$ and $y_2 \notin R$.
Now, the edge $\{x_2,y_2\}$ has the colour associated with $b$ due to (ii), a contradiction.
\end{proof}

Suppose that $R$ is a resolving set of $G$. 
Due to Lemma \ref{lem:colourprops} (i) and (ii), all cliques of $G_R$ are monochromatic.
The subgraph of $G_R$ that consists of only the edges of the colour associated with $r \in R$ is a disjoint union of cliques.

Lemma \ref{lem:denseuniversal} states that universal vertices are not basis forced vertices.
We can obtain the same result using Lemma \ref{lem:colourprops} (iv): if $u$ is a universal vertex of $G$ and $R$ is a metric basis of $G$ that contains $u$, then $d(u,x) = d(u,y)$ for all $x,y \in V(G) \setminus R$ and (iv) cannot be satisfied for $u$.

\begin{thm}\label{thm:nminusbeta}
If $G$ is a graph with $n$ vertices and $k>0$ basis forced vertices, then $k \leq n - \dim (G) - 1$.
Moreover, we have $k \leq \frac{n-1}{2}$.
\end{thm}
\begin{proof}
Let $R$ be a metric basis of $G$.
Denote $U = V \setminus R$ and $m = |U|$.
Due to Lemma \ref{lem:colourprops} (iv), $G_R[U]$ contains at least $k$ differently coloured edges.
Let us consider $k$ differently coloured edges in $G_R[U]$.
If $m \leq k$, then some (or all) of these edges form a cycle that contradicts Lemma \ref{lem:colourprops} (i).
Thus, $m \geq k+1$ and $k \leq m-1 = n - \dim (G) - 1$.
Since $k \leq \dim (G)$, we also have $k \leq \frac{n-1}{2}$.
\end{proof}

We have not found a graph that attains this bound.
However, for the graph $G$ described in Lemma \ref{lem:denseP5} we have $k = n - \dim (G) -2$ when $k=2$.
Indeed, let $G$ be such that $\cg{G} \simeq P_5 \cup \cg{K}_{n-5}$.
According to Lemma \ref{lem:denseP5}, the graph $G$ has two basis forced vertices and $\dim (G) = n-4$.

\begin{thm}\label{thm:genedgebound}
Let $G$ be a connected graph with $n\geq 3$ vertices.
If $G$ has $k > 0$ basis forced vertices, then 
\[ |E(G)| \leq \frac{n(n-1)}{2} - 2k. \]
\end{thm}
\begin{proof}
Let us denote the basis forced vertices of $G$ by $v_1, \ldots , v_k$.
Due to Lemma \ref{lem:denseuniversal}, none of the vertices $v_i$ are universal vertices.
Thus, for each $v_i$ there exists an $x \in V(G)$ such that the edge $v_i x$ is an edge of $\cg{G}$.

Let $R$ be a metric basis of $G$.
Due to Lemma \ref{lem:colourprops} (iv), for any $v_i$ there exist distinct $x,y \in V(G) \setminus R$ such that $d(v_i,x) \neq d(v_i,y)$ and $d(r,x) = d(r,y)$ for all $r \in R \setminus \{v_i\}$.
Thus, at least one of the edges $v_i x$ and $v_i y$ is an edge in $\cg{G}$.
Therefore, every basis forced vertex $v_i$ is not adjacent to some vertex $u$ that is not a basis forced vertex.

Suppose that $v_i u$ is the only edge in $\cg{G}$ from $v_i$ to $V(G) \setminus R$.
Now, $d(v_i,v) = 1$ for all $v \in V(G) \setminus R$ such that $v\neq u$.

As we stated earlier, according to Lemma \ref{lem:colourprops} (iv) there exist distinct $x,y \in V(G) \setminus R$ such that $d(v_i,x) \neq d(v_i,y)$ and $d(r,x) = d(r,y)$ for all $r \in R \setminus \{v_i\}$.
Now, either $x=u$ or $y=u$, since otherwise we have $d(v_i,x) = 1 = d(v_i,y)$, a contradiction.

Suppose without loss of generality $x=u$.
Then $y$ is unique due to Lemma \ref{lem:colourprops} (ii) and the fact that $d(v_i,v) = 1$ for all $v \in V(G) \setminus R$, $v \neq u$. 
According to Lemma \ref{lem:colourprops} (iii) the graph $G_R$ has at least two edges of the colour associated with $v_i$.
The edge $uy$ is one such edge.
However, the edge $uy$ is the only edge within $V(G) \setminus R$ of the colour associated with $v_i$.
Since the set $R$ is independent in $G_R$, there exists a vertex $w \in V(G) \setminus R$ such that the edge $v_i w$ has the colour associated with $v_i$ in $G_R$.

Suppose that $w \in \{u,y\}$.
Then according to Lemma \ref{lem:colourprops} (ii) both edges $v_i u$ and $v_i y$ have the colour associated with $v_i$ in $G_R$.
Now, the three edges $u y$, $v_i u$ and $v_i y$ are the only edges of the colour associated with $v_i$ in $G_R$.
Otherwise, we have a contradiction due to Lemma \ref{lem:colourprops} (ii) and the fact that $d(v_i,v) = 1$ for all $v \in V(G) \setminus R$, $v \neq u$.
Let us show that now $v_i$ is not a basis forced vertex.
If the set $R[v_i \leftarrow u]$ is a metric basis, then we are done.
Hence, suppose that the set $R[v_i \leftarrow u]$ is not a metric basis of $G$.
The only pairs not resolved by $R \setminus \{v_i\}$ are $\{u,y\}$, $\{v_i,u\}$ and $\{v_i,y\}$.
Since $u$ clearly resolves the first two pairs, we have $d(u,v_i)=d(u,y)$.
However, now $d(y,u) \neq d(y,v_i) = 1$, and the set $R[v_i \leftarrow y]$ is a metric basis of $G$.
Thus, $R[v_i \leftarrow u]$ or $R[v_i \leftarrow y]$ is a metric basis of $G$, and $v_i$ is not a basis forced vertex.

Suppose then that $w \notin \{u,y\}$.
Now the vertex $w$ is unique.
Otherwise, we again have a contradiction due to Lemma \ref{lem:colourprops} (ii) and the fact that $d(v_i,v) = 1$ for all $v \in V(G) \setminus R$, $v \neq u$.
Let us show that now $v_i$ is not a basis forced vertex.
The set $R \setminus \{v_i\}$ resolves all but two pairs of vertices: $v_i,w$ and $u,y$.
Now, if some vertex $z \neq v_i$ resolves both of these pairs, then the set $R[v_i \leftarrow z]$ is a metric basis of $G$ and $v_i$ is not a basis forced vertex.
Let us show that such a vertex $z$ exists.

If $d(w,u) \neq d(w,y)$, then the vertex $w$ resolves the two aforementioned pairs and $R[v_i \leftarrow w]$ is a metric basis of $G$.
Suppose that $d(w,u) = d(w,y)$.
Since $d(v_i,w) = 1 = d(v_i,y)$, we have $d(w,y) \leq d(w,v_i) + d(v_i,y) = 2$.
If $d(w,u) = d(w,y) = 1$, then $d(u,v_i) \neq d(u,w)$ and the set $R[v_i \leftarrow u]$ is a metric basis of $G$.
If $d(w,u) = d(w,y) = 2$, then $d(y,v_i) \neq d(y,w)$ and the set $R[v_i \leftarrow y]$ is a metric basis of $G$.
In all cases, the set $R[v_i \leftarrow z]$ is a metric basis of $G$ for some $z \in \{w,u,y\}$, and thus the vertex $v_i$ is not a basis forced vertex of $G$.

In conclusion, for any basis forced vertex $v_i$ there exist at least two edges from $v_i$ to $V(G) \setminus R$ in $\cg{G}$.
Since $G$ has $k$ basis forced vertices, we have $|E(G)| \leq \frac{n(n-1)}{2} - 2k$.
\end{proof}

Let $G$ be a graph such that $\cg{G} \simeq P_5 \cup \cdots \cup P_5 \cup \cg{K}_{m}$, where $P_5$ appears $\frac{k}{2}$ times in the disjoint graph union.
According to Lemma \ref{lem:denseP5} the graph $G$ has $k$ basis forced vertices, two in each complement of $P_5$.
The following corollary is now immediate.

\begin{cor}
For every even positive integer $k$ there exists a graph $G$ with $k$ basis forced vertices and 
\[|E(G)| = \frac{n(n-1)}{2} - 2k \]
where $n$ is the number of vertices of $G$.
\end{cor}

The bound in Theorem \ref{thm:genedgebound} does not seem attainable for odd $k$.
However, we have found a graph family whose members have an odd ($\geq 3$) number of basis forced vertices and $\frac{n(n-1)}{2} - 2k -1$ edges.
Let $H$ be the graph we obtain by attaching a pendant to the middle vertex of the path $v_1 \ldots v_7$.
If $G$ is a graph such that $\cg{G} \simeq H \cup \cg{K}_m$, then the graph $G$ has $\frac{n(n-1)}{2} - 2\cdot 3 -1$ edges and three basis forced vertices: $v_2$, $v_4$ and $v_6$ (this can be shown using similar techniques as in the proof of Lemma \ref{lem:denseP5}).
Moreover, it is possible to show that the graph $G$ for which $\cg{G} \simeq H \cup P_5 \cup \cdots \cup P_5 \cup \cg{K}_m$ where $P_5$ appears $\frac{k-3}{2}$ times has three basis forced vertices in $V(H)$ and two basis forced vertices in each $V(P_5)$.
All in all, the graph $G$ has $3 + 2 \cdot \frac{k-3}{2} = k$ basis forced vertices and $\frac{n(n-1)}{2} - 7 - 4 \cdot \frac{k-3}{2} = \frac{n(n-1)}{2} - 2k -1$ edges.

\section{Algorithmic complexity}\label{sec:Complexity}

In this section, we consider the algorithmic complexity of determining whether a given vertex is a basis forced or void vertex of a metric basis. In particular, we show that the first problem is co-NP-hard and the latter problem is NP-hard. The proofs are based on a polynomial-time reduction from the well-known $3$-SAT problem. Previously, in~\cite{Khuller96}, it has been shown that given an arbitrary graph $G=(V,E)$ and an integer $k$, it is NP-complete to decide whether $\dim(G) \leq k$. The reductions of the proofs of this section are inspired by the one presented in~\cite{Khuller96}. In what follows, this reduction is briefly recapped.

For the $3$-SAT problem, denote the set of variables by $X = \{x_1, x_2, \ldots, x_n\}$ and the set of literals by $U = \{x_1, x_2, \ldots, x_n, \overline{x_1}, \overline{x_2}, \ldots, \overline{x_n}\}$, where $\overline{x_i}$ denotes the negation of the variable $x_i$. Let $F$ be an instance of the $3$-SAT problem; more precisely, let $F$ be a formula $F = \mathcal{C}_1 \wedge \mathcal{C}_2 \wedge \cdots \wedge \mathcal{C}_m$, where each clause $\mathcal{C}_j$ contains exactly three literals, i.e., each clause is of the form $\mathcal{C}_j = u_{j,1} \vee u_{j,2} \vee u_{j,3}$, where $u_{j,1}, u_{j,2}, u_{j,3} \in U$. Based on the given formula $F$, we form a graph $G=(V,E)$ as follows:
\begin{itemize}
\item For each variable $x_i \in X$, we construct a variable gadget of $x_i$ with vertices $a_{i,1}$, $a_{i,2}$, $b_{i,1}$, $b_{i,2}$, $T_i$ and $F_i$ and edges as given in Figure~\ref{fig:variablegadget}.
\item For each clause $\mathcal{C}_j = u_{j,1} \vee u_{j,2} \vee u_{j,3}$, we construct a clause gadget $\mathcal{C}_j$ with vertices $c_{j,1}$, $c_{j,2}$, $c_{j,3}$, $c_{j,4}$, and $c_{j,5}$ and edges as given in Figure~\ref{fig:clausegadget}. Moreover, if $u_{j,k} = x_i$ (where $k=1,2,3$), then $c_{j,3}$ is adjacent to $F_i$, else $u_{j,k} = \overline{x_i}$ and $c_{j,3}$ is adjacent to $T_i$. In addition, $c_{j,3}$ is adjacent to both $T_i$ and $F_i$ for all variables $x_i$ not occurring in $\mathcal{C}_j$ and $c_{j,1}$ is adjacent to both $T_i$ and $F_i$ for any variables $x_i$ occurring in $\mathcal{C}_j$.
\end{itemize}
\begin{figure}
\centering
\begin{subfigure}[b]{0.45\linewidth}
\centering
\begin{tikzpicture}[scale=0.85]
\draw (0,0) -- (1,0) -- (2,-.5) -- (1,.8) -- (0,.8) -- (-1,-.5) -- (0,0);
\draw \foreach \x in {(0,0),(1,0),(0,.8),(1,.8),(-1,-.5),(2,-.5)} {
 	\x node[circle, draw, fill=white,
             inner sep=0pt, minimum width=6pt] {}
};
\draw {
	(0,-.4) node[] {$a_{i,2}$}
	(1,-.4) node[] {$b_{i,2}$}
	(-.5,.8) node[] {$a_{i,1}$}
	(1.5,.8) node[] {$b_{i,1}$}
	(-1.4,-.5) node[] {$T_i$}
	(2.4,-.5) node[] {$F_i$}
};	
\end{tikzpicture}
\caption{The variable gadget of $x_i$.}\label{fig:variablegadget}
\end{subfigure}
\begin{subfigure}[b]{0.45\linewidth}
\centering
\begin{tikzpicture}[scale=0.85]
\draw \foreach \x in {(1,0),(-1,0),(.7,-.7),(-.7,-.7)} {
	(0,0) -- \x };
\draw \foreach \x in {(0,0),(1,0),(-1,0),(.7,-.7),(-.7,-.7)} {
	\x node[circle, draw, fill=white,
                 inner sep=0pt, minimum width=6pt] {}
};
\draw {
	(0,.4) node[] {$c_{j,2}$}
	(1.5,0) node[] {$c_{j,3}$}
	(-1.5,0) node[] {$c_{j,1}$}
	(1.2,-.7) node[] {$c_{j,5}$}
	(-1.2,-.7) node[] {$c_{j,4}$}
};
\end{tikzpicture}
\caption{The clause gadget of $\mathcal{C}_j$.}\label{fig:clausegadget}
\end{subfigure}
\caption{ }
\end{figure}
It is clear that the graph $G$ can be constructed in polynomial time. In~\cite{Khuller96}, it is shown that the formula $F$ is satisfiable if and only if $\dim(G) = n+m$. This implies that the problem of deciding whether $\dim(G) \leq k$ is NP-complete.

Inspired by the previous reduction, we show in the following theorem that determining whether a given vertex is a basis forced one or a void one are algorithmically difficult.
\begin{thm}
Let $G$ be a graph and $u$ be a vertex of $G$.
\begin{itemize}
\item[(i)] Deciding whether $u$ is a basis forced vertex of $G$ is a co-NP-hard problem.
\item[(ii)] Deciding whether $u$ is a void vertex of $G$ is an NP-hard problem.
\end{itemize}
\end{thm}
\begin{proof}
In order to prove the first claim~(i), we show that the problem of deciding whether a given $3$-SAT formula is \emph{not satisfiable} -- a co-NP-complete problem --  can be reduced in polynomial time to the problem of determining if a given vertex is basis forced one of a graph. For the second claim~(ii), we similarly prove that the problem of deciding whether a given $3$-SAT formula is \emph{satisfiable} -- an NP-complete problem -- can be reduced in polynomial time to the problem of determining if a given vertex is basis void one of a graph.

Using the notation introduced above, let $F$ be an instance of the $3$-SAT problem.
Based on the given formula $F$, we form a graph $G'=(V',E')$ as follows:
\begin{itemize}
\item Let $G_1 = (V_1, E_1)$ and $G_2 = (V_2, E_2)$ be two disjoint copies of the graph $G$ constructed in the reduction of~\cite{Khuller96} (with respect to $F$). Denote the vertices of the graphs by $a^k_{i,1}$, $a^k_{i,2}$, $b^k_{i,1}$, $b^k_{i,2}$, $T^k_i$, $F^k_i$, $c^k_{j,1}$, $c^k_{j,2}$, $c^k_{j,3}$, $c^k_{j,4}$ and $c^k_{j,5}$, where $k \in \{1,2\}$.
\item Furthermore, an edge is added from each $c^1_{j,1}$ and $c^1_{j,3}$ to $T^2_i$ and $F^2_i$ for all $i \in \{1, \ldots, n\}$. Analogously, an edge is added from each $c^2_{j,1}$ and $c^2_{j,3}$ to $T^1_i$ and $F^1_i$ for all $i \in \{1, \ldots, n\}$.
\item Finally, add a vertex $w$ such that it is adjacent to $c^k_{j,3}$ for all $k = 1, 2$ and $j = 1, \ldots, m$.
\end{itemize}
It is immediate that the graph $G' = (V', E') = (V_1 \cup V_2 \cup \{w\}, E')$ can be constructed in polynomial time. Before diving into the proofs of~(i) and~(ii), we need to discuss some preliminary results.

Let $R$ be a resolving set of $G'$. We first present the following simple observations, which mimic the ones given in Lemmas~A.2 and A.3 of~\cite{Khuller96}:
\begin{itemize}
\item[(a)] Observe that for $k = 1, 2$ and $i = 1, \ldots, n$, at least one of the vertices $a^k_{i,1}$, $a^k_{i,2}$, $b^k_{i,1}$ and $b^k_{i,2}$ belongs to the resolving set $R$. Indeed, suppose to the contrary that $R \cap \{a^k_{i,1}, a^k_{i,2}, b^k_{i,1}, b^k_{i,2}\} = \emptyset$. This implies that $a^k_{i,1}$ and $a^k_{i,2}$ are not resolved (a contradiction) since the corresponding variable gadget is connected to the rest of the graph only through the vertices $T^k_i$ and $F^k_i$.
\item[(b)] Observe that for $k = 1, 2$ and $j = 1, \ldots, m$, at least one of the vertices $c^k_{j,4}$ and $c^k_{j,5}$ belongs to the resolving set $R$. Indeed, if $R \cap \{c^k_{j,4}, c^k_{j,5}\} = \emptyset$, then $c^k_{j,4}$ and $c^k_{j,5}$ are clearly not resolved (a contradiction).
\end{itemize}
By the observations, we immediately obtain that the metric dimension of $G'$ is at least $2n+2m$. In what follows, we show that $\dim(G') = 2n + 2m$ if and only if the formula $F$ is satisfiable.

Let us first show that if the formula $F$ is satisfiable, then the metric dimension $\dim(G') = 2n+2m$. Let $A$ be a satisfiable truth assignment of $F$. Construct then a set $C$ as follows: $C$ consists of all the vertices $c^k_{j,4}$, where $k = 1,2$ and $j = 1, \ldots, m$, and if the assignment of $x_i$ is \emph{true} in $A$, then $a^1_{i,1}$ and $a^2_{i,1}$ belong to $C$, and otherwise $b^1_{i,1}$ and $b^2_{i,1}$ are in $C$. We immediately notice that $C$ contains exactly $2n+2m$ vertices. In what follows, we show that $C$ is a resolving set of $G'$. For this purpose, we first observe that all the pairs $x,y \in V'$ of distinct vertices except $c^k_{j,1}$ and $c^k_{j,3}$ are resolved by some element of $C$:
\begin{itemize}
\item Suppose first that $x = c^k_{j,l}$ with $l \in \{1, \ldots, 5\}$ and $y$ is any other vertex not in the same clause gadget as $x$. Then $d(c^k_{j,4}, x) \leq 2$ and $d(c^k_{j,4}, y) > 2$, and we are immediately done. Furthermore, if $y$ belongs to the same clause gadget as $x$, then some element of $C$ clearly resolves $x$ and $y$ unless the vertices are $c^k_{j,1}$ and $c^k_{j,3}$.
\item Suppose then that $x = w$. Now the distance of $x$ to any vertex of $C$ is exactly three. Clearly, this is not the case for any other vertex $y$ and we are done.
\item Finally, suppose that $x$ belongs to some variable gadget. Without loss of generality, we may assume that $x$ belongs to the copy of the gadget of $x_1$ in $G_1$ and that $a^1_{1,1} \in C$. Now it is immediate that if $y$ is in the same variable gadget as $x$, then either $a_{1,1}^1$ or any $c_{j,4}^2$ resolves $x$ and $y$. By the previous cases, we are also immediately done if $y$ does not belong to any variable gadget. Hence, we may assume that $y$ belongs to some variable gadget other than the one of $x_1$. Now, if $x \neq b^1_{1,2}$, then $d(a^1_{1,1},x) \leq 2$ and $d(a^1_{1,1},y) > 2$. Furthermore, if $x = b^1_{1,2}$ and $d(a^1_{1,1},x) = d(a^1_{1,1},y)$, then $d(a^1_{1,1},x) = d(a^1_{1,1},y) = 3$ and $y$ is equal to $T^k_i$ or $F^k_i$ for some $k$ and $i$. This further implies that there exists a vertex $c \in C$ in the same variable gadget as $y$ such that $d(c,y) \leq 2$ and $d(c,x) > 2$. Hence, we are done.
\end{itemize}
Thus, it is enough to consider the pairs of vertices $x  = c^k_{j,1}$ and $y  = c^k_{j,3}$. It is immediate that the distance of $c^k_{j,1}$ to any vertex of $C$ in the variable gadgets corresponding to $u_{j,1}$, $u_{j,2}$ and $u_{j,3}$ is equal to $2$. However, by the fact that $A$ is a satisfiable truth assignment of $F$ and the construction of $C$, the distance of $c^k_{j,3}$ to some vertex of $C$ in the variable gadgets corresponding to $u_{j,1}$, $u_{j,2}$ and $u_{j,3}$ is equal to $3$. Thus, in conclusion, $C$ is a resolving set of $G'$ and $\dim(G') = 2n+2m$.

Let us then show that if the metric dimension of $G'$ is $2n+2m$, then the formula $F$ is satisfiable. Let $C$ be a resolving set of $G'$ with $2n+2m$ vertices. Due to the observations~(a) and (b), we know that for each $k \in \{1,2\}$ and $i \in  \{1, \ldots, n\}$ exactly one of the vertices $a^k_{i,1}$, $a^k_{i,2}$, $b^k_{i,1}$ and $b^k_{i,2}$ belongs to $C$ and for each $k \in \{1,2\}$ and $j \in  \{1, \ldots, m\}$ exactly one of the vertices $c^k_{j,4}$and $c^k_{j,5}$ belongs to $C$. Form then a truth assignment $A$ of $F$ as follows: if $a^1_{i,1}$ or $a^1_{i,2}$ belongs to $C$, then set the variable $x_i$ to be \emph{true}, else ($b^1_{i,1}$ or $b^1_{i,2}$ belongs to $C$ and) set $x_i$ to be \emph{false}. In what follows, we show that the truth assignment $A$ satisfies the formula $F$. Suppose to the contrary that a clause $\mathcal{C}_j$ is not satisfied by $A$. This implies that the distance of $c^1_{j,1}$ and $c^1_{j,3}$ to any vertex of $C$ in the variable gadgets corresponding to $u_{j,1}$, $u_{j,2}$ and $u_{j,3}$ is equal to $2$. Furthermore, it is straightforward to verify that the distance of $c^1_{j,1}$ and $c^1_{j,3}$ to any vertex of $C$ in other variable gadgets is also equal to $2$ and that their distances are  equal to $4$ to any other vertices of $C$ (in the clause gadgets). Thus, they are not resolved by any element of $C$ and a contradiction follows. Thus, the truth assignment $A$ satisfies the formula $F$.

Let us next show that $\dim(G') \leq 2n+2m + 1$ regardless of the existence of a satisfiable truth assignment for the formula $F$. For this purpose, let $C$ be a set consisting of $w$ and all the vertices $a^k_{i,1}$ and $c^k_{j,4}$, where $k = 1,2$, $i = 1, \ldots, n$ and $j = 1, \ldots, m$. Clearly, the cardinality of $C$ is equal to $2n+2m+1$. As above, it can be shown that a pair of distinct vertices is resolved even without taking into account the vertex $w$ unless the pair is $c^k_{j,1}$ and $c^k_{j,3}$. However, it is immediate that they are resolved since $d(w, c^k_{j,1}) = 3 \neq 1 = d(w, c^k_{j,3})$. Therefore, the set $C$ is a resolving set of $G'$ and $\dim(G') \leq 2n+2m + 1$.

(i) Now we are ready to prove that $w$ is a basis forced vertex of $G'$ if and only if the formula $F$ is not satisfiable. Observe first that if $w$ is a basis forced vertex of $G'$, then by the observations~(a) and (b), we obtain that $\dim(G') \geq 2n+2m + 1$. Therefore, we have $\dim(G') = 2n+2m + 1$ and $F$ is not satisfiable. For the other direction, assume that $F$ is not satisfiable. Hence, we have $\dim(G') \neq 2n+2m$ implying $\dim(G') =2n+2m + 1$. Suppose to the contrary that $w$ is not a basis forced vertex of $G'$ and there exists a metric basis $C$ of $G'$ with $2n+2m + 1$ vertices such that $w \notin C$. By the observations~(a) and (b), we obtain that either $G_1$ or $G_2$ contains exactly $n+m$ vertices of $C$; without loss of generality, we may assume that $G_1$ is such a graph. Notice that $c^1_{j,1}$ and $c^1_{j,3}$ are resolved for all $j = 1, \ldots, m$ since $C$ is a metric basis of $G'$. As above, it can now be shown that $F$ is satisfiable (a contradiction). Therefore, $w$ is a basis forced vertex of $G'$. Thus, there exists a polynomial-time reduction of the complement of the $3$-SAT problem to the problem of deciding whether a given vertex is a basis forced vertex. Hence, the studied problem is co-NP-hard.

(ii) Let us then show that $w$ is a void vertex of $G'$ if and only if the formula $F$ is satisfiable. Observe first that if $F$ is satisfiable, then $\dim(G') = 2n+2m$. Hence, if $C$ is any metric basis of $G'$ (with cardinality $2n+2m$), then $w$ does not belong to $C$ by the observations~(a) and (b). Therefore, $w$ is a void vertex of $G'$. For the other direction, assume that $w$ is a void vertex of $G'$. Recall that $F$ is satisfiable if and only if $\dim(G') = 2n+2m$. Suppose to the contrary that $F$ is not satisfiable, i.e., $\dim(G') \neq 2n+2m$. This implies that $\dim(G') = 2n+2m+1$. Let $C$ be a metric basis of $G'$ (with cardinality $2n+2m+1$ and $w \notin C$). As above, we obtain that $G_1$ or $G_2$ contains exactly $n+m$ vertices of $C$; without loss of generality, we may assume that $G_1$ is such a graph. Analogously, as in the case~(i), we can show that $F$ is satisfiable since $w \notin C$ (a contradiction). Hence, if $w$ is a void vertex, then $F$ is satisfiable. Thus, there exists a polynomial-time reduction of the $3$-SAT problem to the problem of deciding whether a given vertex is a void vertex. Hence, the studied problem is NP-hard. 
\end{proof}

\section{Future Works}

Here are some open problems related to the questions in this paper.

\begin{itemize}
\item Find graphs with $k$ basis forced vertices but fewer edges than in our construction in Section \ref{sec:Sparse}.

\item Find a characterisation of graphs with $k$ basis forced vertices and metric dimension $k$.

\item Improve the bound $k \leq n - \dim (G) - 1$ in Theorem \ref{thm:nminusbeta} or find a graph attaining it.
\end{itemize}

\section*{Acknowledgements}

The last author (Ismael G. Yero) has been partially supported by the Spanish Ministry of Science and Innovation through the grant PID2019-105824GB-I00.


\end{document}